\documentclass{article}
\usepackage[utf8]{inputenc}
\usepackage{amsmath,amssymb,amsthm,amsfonts,latexsym,graphicx,mathtools,enumitem,bbm}

\usepackage{tikz}
\usepackage{subcaption}

\usepackage{url}
\usepackage{xcolor}
\usepackage{svg}

\usepackage{scalerel}

\newcommand{\R}{\mathbb{R}}
\DeclareMathOperator{\RQ}{RQ}

\DeclareMathOperator{\id}{Id}

\DeclareMathOperator{\supp}{supp}
\DeclareMathOperator{\vol}{vol}

\newcommand{\tilh}[0]{\tilde{h}}
\DeclareMathOperator{\range}{Range}

\theoremstyle{plain}
\newtheorem{theorem}{Theorem}[section]
\newtheorem{lemma}[theorem]{Lemma}
\newtheorem{proposition}[theorem]{Proposition}
\newtheorem{corollary}[theorem]{Corollary}

\newtheorem{question}{Question}
\newtheorem{claim}{Claim}

\theoremstyle{definition}
\newtheorem{definition}[theorem]{Definition}
\newtheorem{example}[theorem]{Example}
\theoremstyle{remark}
\newtheorem{remark}[theorem]{Remark}

\usepackage{stmaryrd} %for double square brackets
\usepackage[affil-it]{authblk}
\usepackage{booktabs}   % for \toprule et al
\usepackage[numbers,sort&compress]{natbib}

\usepackage{hyperref}
\usepackage{comment}

\usepackage[margin=3.5cm]{geometry}

\title{Cheeger-type inequalities for the second largest spectral gap from $1$ of the normalized Laplacian}

\author[1]{Lies Beers\thanks{e.g.m.beers@vu.nl}}
\author[1]{Raffaella Mulas}
\author[2]{Jan Petr}
\affil[1]{Vrije Universiteit Amsterdam, Amsterdam, The Netherlands}
\affil[2]{University of Passau, Passau, Germany}

\date{}

\allowdisplaybreaks

\begin{document}

\maketitle

\begin{abstract}
   We study the second largest spectral gap from $1$ of the normalized Laplacian of a graph, a quantity that appears in the literature in connection with random walks, expander graphs, and Ramanujan graphs. We relate it to the classical Cheeger and dual Cheeger constants, and we introduce a new Cheeger-type constant admitting a probabilistic interpretation in terms of two-step random walks. For this constant, we establish sharp inequalities analogous to the classical Cheeger inequalities.
\end{abstract}

\textbf{Note added in v2.} A referee pointed out that the bounds in Corollary \ref{cor:lower} and Theorem \ref{thm:upper} follow from the Cheeger inequalities for weighted graphs \cite{BauerJost2013}. Indeed, the matrix $M$ from \eqref{eq:defM} is the transition matrix of the random walk on the weighted graph with adjacency matrix $DM=AD^{-1}A$, whose Cheeger constant is $1-\tilde{h}$, while the second largest eigenvalue of $M$ is $\tau^2$. We are grateful to the referee for this observation, which unfortunately means that these inequalities (as well as most results from Section \ref{section:tilde h}) are not new. Nevertheless, our direct and self-contained proof of Theorem \ref{thm:upper} yields a characterization of its equality case which, to the best of our knowledge, is new. We also believe that the probabilistic interpretation of $\tilde{h}$ and the formulation of the inequalities in terms of $\tau$ may still be of interest. In light of this, we decided to keep the paper on arXiv, but not to submit it to a journal.

\section{Introduction}\label{section:intro}

Let $G$ be a simple graph on $n$ nodes with no isolated vertices, and let 

$$
0= \lambda_1\leq  \lambda_2\leq \ldots \leq \lambda_n \leq 2 
$$
denote the eigenvalues of its normalized Laplacian. Since the foundational work of Fan Chung \cite{chung}, these eigenvalues have been extensively studied and are known to encode important geometric and combinatorial properties of graphs. In particular, the classical Cheeger inequalities relate the eigenvalue $\lambda_2$ to the Cheeger constant $h$ \cite{chung}, while the largest eigenvalue $\lambda_n$ is linked to the dual Cheeger constant $\overline{h}$ through the dual Cheeger inequalities \cite{Bauer14,Chang16c,BauerJost2013}.

In \cite{JMZ23}, the smallest spectral distance from $1$, 
$$
\varepsilon:=\min_{i\neq 1} |1-\lambda_i|
$$ was investigated, and the sharp bound $\varepsilon\leq 1/2$ was established for connected graphs on $n\geq 3$ vertices. Here, we study instead the second largest spectral distance from $1$,
$$
\tau:=\max_{i\neq 1} |1-\lambda_i|,
$$
which controls the rate of convergence of a random walk on $G$ to its stationary distribution, as explained in Remark~\ref{rmk:walks}. This quantity also appears in Fan Chung's book \cite[Chapters 1.5, 5, and 6]{chung}, in relation to random walks and expander graphs.

Among other results, in Proposition~\ref{prop:char} we show that $$\tau=\max\{1-\lambda_2,\lambda_n-1\},$$ and as a consequence we relate $\tau$ to the classical Cheeger and dual Cheeger constants in Proposition~\ref{prop:hat}. We further introduce a new Cheeger-type constant $\tilde{h}$ which admits an interpretation in terms of the probability of returning to a set of vertices after two steps (Definition~\ref{def:tildeh} and Remark~\ref{rmk:geom0}) for which we establish the following sharp inequalities (Corollary~\ref{cor:lower} and Theorem~\ref{thm:upper}):

\begin{equation*}
2\tilde{h}-1  \leq   \tau^2\leq \sqrt{1-(1-\tilh)^2}.
\end{equation*}

Our work is also related to Ramanujan graphs \cite{lubotzky1988ramanujan,murty2003ramanujan}. In fact, if $G$ is $d$-regular, then the normalized Laplacian eigenvalues satisfy $\lambda_i=1-\mu_i/d$,
 where $\mu_i$ are the eigenvalues of the adjacency matrix. Since, by definition, a Ramanujan graph is a $d$-regular graph satisfying
 $$\max_{i\,: \,\mu_i\neq \pm d}|\mu_i|\leq 2\sqrt{d-1},
 $$
Ramanujan graphs can equivalently be characterized as $d$-regular graphs such that
 $$
 \max_{i\,: \,\lambda_i\neq 0,2}|1-\lambda_i|\leq \frac{2\sqrt{d-1}}{d}.
 $$
For non-bipartite connected graphs, this is equivalent to

\begin{equation}\label{eq:Ramanujan}
    \tau\leq \frac{2\sqrt{d-1}}{d}.
\end{equation}

\subsection*{Structure of the paper}

The paper is organized as follows. In Section~\ref{section:Background}, we recall the necessary background on the normalized Laplacian and the relevant eigenvalue inequalities. In Section~\ref{section:tau}, we introduce $\tau$ and $\tilde{h}$, and we relate $\tau$ to the classical Cheeger and dual Cheeger constants. In Section~\ref{section:tilde h} we study basic properties of $\tilde{h}$ and pose two questions for future research. We then derive lower bounds for $\tau^2$
 in Section~\ref{section:lower}, and upper bounds in Section~\ref{section:upper}. Since the proof of the main upper bound is rather long, we dedicate Section~\ref{section:proof} entirely to it. 

\section{Background}\label{section:Background}

Fix a simple graph $G=(V,E)$ on $n$ nodes and $m$ edges, with no isolated vertices, and let $v_1,\ldots,v_n$ denote the vertices of $G$. The \emph{adjacency matrix} of $G$ is the $n\times n$ matrix $A:=A(G)$ with entries 
\begin{equation*}
    A_{ij}:=\begin{cases}1 &\text{ if }v_i\sim v_j\\
    0 &\text{ otherwise.}\end{cases}
\end{equation*} The \emph{normalized Laplacian} of $G$ is the matrix $$L:=L(G):=\id-D^{-1}A,$$
 where $\id$ denotes the $n\times n$ identity matrix, and $D:=D(G):=\textrm{diag}(\deg v_1,\ldots,\deg v_n)$ is the \emph{degree matrix} of $G$. We denote the eigenvalues of $L$ by
\begin{equation*}
    0=\lambda_1\leq \lambda_2\leq \ldots\leq \lambda_n\leq 2.
\end{equation*}

Now, let $C(V)$ denote the vector space of functions $f:V\rightarrow\mathbb{R}$, and, given $f,g\in C(V)$, let
    \begin{equation}\label{eq:<>}
	\langle f,g\rangle:=\sum_{v\in V}\deg v\cdot f(v)\cdot g(v).
	\end{equation}
We can see then see the normalized Laplacian $L$ as an operator $C(V)\rightarrow C(V)$ such that
\begin{equation*}
    Lf(v)=f(v)-\frac{1}{\deg v}\sum_{w\sim v}f(w).
\end{equation*}Moreover, $L$ is self-adjoint with respect to the inner product $	\langle \cdot,\cdot \rangle$, i.e.,
	\begin{equation*}
	\langle Lf,g\rangle=\langle f,Lg\rangle \quad \forall f,g\in C(V).
	\end{equation*}
    
In analogy with \cite{JMZ23} (proof of Lemma 3), we note that the values $(1-\lambda_1)^2,\ldots,(1-\lambda_n)^2$ are precisely the eigenvalues of the matrix
\begin{equation}\label{eq:defM}
    M:=(\id-L)^2=(D^{-1}A)^2,
\end{equation}
whose entries are 
\begin{equation}\label{eq:Mentries}
M_{vu}=\sum_{w\in V}(D^{-1}A)_{vw}(D^{-1}A)_{wu}=\sum_{w\in \mathcal N(v)\cap \mathcal N(u)}
\frac{1}{\deg v\,\deg w},\quad\text{ for }v,u\in V.
\end{equation}

In particular, $\varepsilon^2$ is the smallest eigenvalue of $M$, while $\tau^2$ is the second largest. Moreover, similarly to $L$, $M$ is also self-adjoint with respect to the weighted inner product in \eqref{eq:<>}, implying that its eigenvalues can be characterized using the Courant--Fischer--Weyl min-max Principle via the Rayleigh Quotients
\begin{equation*}
    \textrm{RQ}(f):=\frac{\langle Mf,f\rangle}{\langle f,f \rangle}=\frac{\sum\limits_{w\in V}\frac{1}{\deg w}\left(\sum\limits_{v\in \mathcal{N}(w)}f(v)\right)^2}{\sum\limits_{w\in V} \deg w\cdot f(w)^2},
\end{equation*}
for functions $f\in C(V) \setminus \{\mathbf 0\}$. In particular,

$$
\varepsilon^2= \min_{f\in C(V)\setminus\{ \mathbf{0}\}} \RQ(f),
$$

while

\begin{equation}\label{eq:tau2}
    \tau^2= \max_{\substack{f\in C(V)\setminus\{\mathbf{0}\}:\\\sum_{w\in V}\deg w\cdot f(w)=0}} \RQ(f).
\end{equation}

\begin{remark}\label{rmk:M}
The approach adopted here differs slightly from that in \cite{JMZ23}. In particular, instead of working with the matrix $(D^{-1/2}AD^{-1/2})^2$, here we consider $M=(D^{-1}A)^2$, together with the weighted inner product in \eqref{eq:<>}. This allows us to work within the same framework as $L$, as they are self-adjoint with respect to the same inner product. Nevertheless, the characterization obtained above for $\varepsilon^2$ is equivalent to that of Lemma 3 in \cite{JMZ23}. The characterization of $\tau^2$ can be derived analogously by observing that $\tau^2$ is the second largest eigenvalue of $M$, while the eigenfunctions corresponding to the largest eigenvalue of $M$ (which is $(1-\lambda_1)^2=1$) are precisely the constant functions.  Hence, by the Courant--Fischer--Weyl min-max Principle, $\tau^2$ is obtained by maximizing the Rayleigh Quotient over all functions that are orthogonal to the constant functions with respect to the inner product in \eqref{eq:<>}, namely over all $f\in C(V)$ such that
\[
\langle f,\mathbf{1}\rangle
=
\sum_{w\in V}\deg w \cdot f(w)
=
0.
\]
\end{remark}

\begin{remark}\label{rmk:walks}
    Let $T:=D^{-1}A$. Then, 
$$
T_{vu}=\begin{cases}\frac{1}{\deg v} &\text{ if }v\sim u\\
    0 &\text{ otherwise.}\end{cases}
$$

Therefore, $T$ is the transition probability matrix of a random walk on $G$, while $M=T^2$ is the two-step transition probability matrix of the random walk. In particular, $M_{vu}$ in \eqref{eq:Mentries} is the probability that a random walk starting at $v$ is at $u$ after two steps. Moreover, since $\tau$ is the absolute value of the second largest eigenvalue of $T$, it controls the rate of convergence of a random walk on $G$ to its stationary distribution (see for instance Section 2.2 in \cite{guruswami2016rapidly}, Section 12.2 in \cite{levin2017markov}, or Section 9 in \cite{BauerJost2013}). 
\end{remark}

We now recall the classical geometric quantities associated with $\lambda_2$ and $\lambda_n$, namely the Cheeger and dual Cheeger constants, together with the corresponding inequalities relating them to $\lambda_2$ and $\lambda_n$. To this end, we first introduce some preliminary notation.

\begin{definition}
Let $U,S\subseteq V$. We denote by
\[
E(U,S):=\Bigl\{\{u,s\}\in E:u\in U,\ s\in S\Bigr\}
\]
the set of edges between $U$ and $S$, and we let
\[
e(U,S):=|E(U,S)|.
\]
Similarly, for $w\in V$, we define $E(w,S):=E(\{w\},S)$ and we let $e(w,S):=|E(w,S)|$ denote its cardinality. Moreover, we let $\bar S:=V\setminus S$ denote the \emph{complement} of $S$, and we let 
$$
\vol S :=\sum_{v\in S}\deg v
$$
be the \emph{volume} of $S$. Finally, the \emph{neighborhood} of $S$ is defined as
\[
\mathcal{N}(S):=\{w\in V:E(w,S)\neq\emptyset\}.
\]
\end{definition}

We can now define the Cheeger and dual Cheeger constants.

\begin{definition}\label{def:h}
The \emph{Cheeger constant} of $G$ is

\begin{equation*}
    h:=h(G):=\min_{\emptyset\neq S\subsetneq V}\frac{e(S,\overline{S})}{\min\{\vol S,\vol \overline{S}\}}.
\end{equation*}
    The \emph{dual Cheeger constant} of $G$ is
	   \begin{equation*}
	       \bar{h}:= \bar{h}(G):=\max_{\substack{\text{partitions}\\ V=S_1\sqcup S_2 \sqcup S_3\, :\\ S_1,S_2\neq\emptyset}} \frac{2\cdot e(S_1,S_2)}{\vol S_1+\vol S_2}.
	   \end{equation*}
\end{definition}

\begin{remark}
A set $S$ realizing the Cheeger constant corresponds to a sparse edge cut $E(S,\overline{S})$ which separates the graph into two parts of comparable volume. Thus, small values of $h$ indicate that the graph is close to being disconnected. Moreover, we always have $0\leq h\leq 1$, with $h=0$ if and only if $G$ is disconnected.

 Similarly, a partition $V=S_1\sqcup S_2\sqcup S_3$ realizing the dual Cheeger constant corresponds to a large set of edges between $S_1$ and $S_2$, with relatively few edges inside $S_1$ and $S_2$. Thus,  large values of $\bar h$ indicate that a large portion of the graph is close to bipartite. Moreover, by Theorem 3.1 in \cite{BauerJost2013}, since we are not admitting isolated vertices, we always have $0< \bar h\leq 1$, and $\bar h=1$ if and only if $G$ is bipartite.
\end{remark}

The Cheeger inequalities \cite{chung,book_Jost_Mulas_Zhang} state that
\begin{equation}\label{eq:Cheeger}
1-\sqrt{1-h^2} \leq \lambda_2 \leq 2h,
\end{equation}

while the dual Cheeger inequalities \cite{Bauer14,Chang16c,BauerJost2013} state that

\begin{equation}\label{eq:Dual-Cheeger}
2\overline h \leq \lambda_n \leq 1 + \sqrt{1-\left(1-\overline h\right)^2}.
\end{equation}

We conclude this section by introducing additional notation that will be used throughout the paper.

\begin{definition}\label{def:indicator}
    Let $U\subseteq V$. The \emph{indicator function} $\mathbbm{1}_U\colon V\to \R$ is defined as
    \[
    \mathbbm{1}_U(v)\coloneqq \begin{cases}
        1, &\text{ if } v \in U,\\
        0, &\text{ if } v \notin U.
    \end{cases}
    \]
\end{definition}

\section{The quantities \texorpdfstring{$\tau$}{t} and \texorpdfstring{$\tilde{h}$}{h}}\label{section:tau}

In this section, we study basic properties of $\tau$ and introduce the new Cheeger-type constant $\tilde{h}$. We again let $G=(V,E)$ be a simple graph on $n$ nodes with no isolated vertices, and we let 
$$
0= \lambda_1\leq  \lambda_2\leq \ldots \leq \lambda_n \leq 2 
$$
denote the eigenvalues of its normalized Laplacian $L$. Recall that, in the Introduction, we defined
$$
\tau:=\max_{i\neq 1} |1-\lambda_i|.
$$
Moreover, in \eqref{eq:tau2},  we established the characterization 
$$
\tau^2= \max_{\substack{f\in C(V):\\\sum_{w\in V}\deg w\cdot f(w)=0}} \RQ(f).
$$

We also note that an estimate for $\tau$ was obtained in \cite[Proposition 2]{JMZ23}, where it was shown that
\[
\frac{1}{n-1}\leq \tau\leq 1.
\]

We now characterize $\tau$ as follows.

\begin{proposition}\label{prop:char}
    For every graph $G$, $$\tau=\max\{1-\lambda_2,\lambda_n-1\}.$$
\end{proposition}

\begin{proof}
    By definition of $\tau$, and since $0=\lambda_1\leq \lambda_2\leq \ldots\leq \lambda_n\leq 2$, we have that
    $$
    \tau=\max \left\{|1-\lambda_2|,\ldots,|1-\lambda_n|\right\}=\max \left\{|1-\lambda_2|,|1-\lambda_n|\right\}.
    $$

Moreover, since $\sum_{i=1}^n \lambda_i=n$ for all graphs, it follows that $\lambda_n>1$. In fact, if $\lambda_n\leq 1$, then $\lambda_i\leq 1$ for all $i$, and therefore

$$
\sum_{i=1}^n \lambda_i=\sum_{i=2}^n \lambda_i\leq n-1<n,
$$ which is a contradiction. Hence,
\begin{equation}\label{eq:AA}
    \tau=\max \left\{|1-\lambda_2|,\lambda_n-1\right\}.
\end{equation}
 We now consider two cases.

\begin{itemize}
    \item Case 1: $\lambda_2\leq 1$. In this case, $|1-\lambda_2|=1-\lambda_2$, and therefore, by \eqref{eq:AA},
    $$\tau=\max\{1-\lambda_2,\lambda_n-1\}.$$
    \item Case 2: $\lambda_2 > 1$. In this case, $|1-\lambda_2|=\lambda_2-1\leq \lambda_n-1$. Therefore, by \eqref{eq:AA},
    $$\tau=\max\{|1-\lambda_2|,\lambda_n-1\}=\lambda_n-1.$$
    Moreover, since $1-\lambda_2< 0$ in this case, it is also true that
    $$
    \tau=\lambda_n-1=\max\{1-\lambda_2,\lambda_n-1\}.
    $$
\end{itemize}

This proves the claim.
    
\end{proof}

Now, let $\widehat{h}:=\min\left\{h,1-\bar h\right\}$. Combining the Cheeger inequalities in \eqref{eq:Cheeger} 
with the dual Cheeger inequalities in \eqref{eq:Dual-Cheeger}, we can infer the following.

\begin{proposition}\label{prop:hat}
 For every graph,
$$
1-\sqrt{1-\widehat{h}^2}\leq 1-\tau\leq 2\widehat{h},
$$
or equivalently,

$$
1-2\widehat{h}\leq \tau \leq \sqrt{1-\widehat{h}^2}.
$$

\end{proposition}

The proof of Proposition~\ref{prop:hat} uses the following observations.

\begin{remark}\label{rmk:min}
    If $a\leq A$ and $b\leq B$, then $ \min\{a,b\}\leq a\leq A$ and $ \min\{a,b\}\leq b\leq B$, therefore
    $$
    \min\{a,b\}\leq \min\{A,B\}.
    $$
\end{remark}

\begin{remark}\label{rmk:min2}
If a function $f$ is increasing on an interval $I$,  then for all $a,b\in I$,

$$
\min\{f(a),f(b)\}=f(\min\{a,b\}).
$$

\end{remark}

\begin{proof}[Proof of Proposition~\ref{prop:hat}]
Since the two statements are clearly equivalent, it suffices to prove the first one. By Proposition~\ref{prop:char},
  $\tau=\max\{1-\lambda_2,\lambda_n-1\}$,  implying that
  $$1-\tau=\min\{\lambda_2,2-\lambda_n\}.$$

Moreover, by \eqref{eq:Dual-Cheeger}, we have that

\begin{equation}\label{eq:Dual-Cheeger2}
1 - \sqrt{1-\left(1-\overline h\right)^2}\leq 2-\lambda_n \leq 2(1-\overline h).
\end{equation}

By \eqref{eq:Cheeger} and \eqref{eq:Dual-Cheeger2}, together with Remark~\ref{rmk:min}, we infer that

$$
1-\tau= \min\{\lambda_2,2-\lambda_n\}\leq \min\{2h,2(1-\overline h)\}=2\widehat{h}
$$

and

$$
1-\tau= \min\{\lambda_2,2-\lambda_n\}\geq \min\left\{1-\sqrt{1-h^2},1 - \sqrt{1-\left(1-\overline h\right)^2}\right\}.
$$

Thus, it remains to show that
$$
\min\left\{1-\sqrt{1-h^2},1 - \sqrt{1-\left(1-\overline h\right)^2}\right\}=1-\sqrt{1-\widehat{h}^2},
$$  This follows from Remark~\ref{rmk:min2}, together with the fact that the function $$f(x):=1-\sqrt{1-x^2}$$ is increasing on $[0,1]$, which implies that

$$
\min\{f(h),f(1-\overline h)\}=f(\min\{h,1-\overline h\})=f(\widehat{h}).
$$

\end{proof}

Proposition~\ref{prop:hat} bounds $\tau$ in terms of $\widehat{h}:=\min\left\{h,1-\bar h\right\}$, and therefore in terms of both $h$ and $\bar h$. Our goal is instead to establish sharp inequalities for $\tau$ (or $\tau^2$) involving a single geometric quantity. To this end, we introduce the following new Cheeger-type constant.

\begin{definition}\label{def:tildeh}
    Given a nonempty set $S\subseteq V$, we let

\begin{equation*}
    \tilde{h}(S):=\frac{1}{\vol S}\cdot \sum_{w\in V}\frac{e(w, S)^2}{\deg w}=\frac{1}{\vol S}\cdot \sum_{w\in \mathcal{N}(S)}\frac{e(w, S)^2}{\deg w}.
\end{equation*}

Moreover, we let

\begin{equation*}
\tilde{h}:=\max_{\substack{\emptyset\neq S \subseteq V:\\ \vol S\leq \vol V/2}} \tilde{h}(S).
\end{equation*}

\end{definition}

\begin{remark}\label{rmk:geom0}
 The quantity $\tilde{h}(S)$ becomes large when the vertices in $\mathcal{N}(S)$ devote a large proportion of their degrees to edges connecting to $S$. Moreover,
$$
\tilde{h}(S)=\sum_{w\in V}\frac{e(w, S)}{\vol S}\cdot \frac{e(w, S)}{\deg w},
$$
 where $e(w,S)/\vol S$ is the probability that a one-step random walk starting from $S$ reaches $w$, assuming that the starting vertex in $S$ is selected with probability proportional to its degree,  while $e(w,S)/\deg w$ is the probability of returning from $w$ to $S$ in one step. Therefore, the quantity $\tilde{h}(S)$ can be interpreted as the probability that a random walk starting in $S$ returns to $S$ after two steps. This probabilistic interpretation suggests a deep connection between $\tilde h$ and $\tau^2$ in view of Remark~\ref{rmk:walks}. Moreover, the Cheeger constant of $G$ (Definition \ref{def:h}) can be written as
\begin{equation*}
    h=\min_{\substack{\emptyset\neq S \subseteq V:\\ \vol S\leq \vol V/2}}
    h(S),
\end{equation*}where
\begin{equation*}
    h(S):=
    \frac{e(S,\overline{S})}{\vol S}
\end{equation*}
is the probability that a one-step random walk starting from $S$ leaves $S$, and therefore $1-h(S)$ is the probability that a one-step random walk starting from $S$ remains in $S$. Hence, $\tilde h$ may be seen as a two-step analogue of $1-h$.
\end{remark} 

In Section~\ref{section:tilde h}, we investigate several properties of $\tilde{h}$. Sections~\ref{section:lower},~\ref{section:upper} and~\ref{section:proof} are devoted to proving the following sharp inequalities:

\begin{equation*}
2\tilde{h}-1  \leq   \tau^2\leq \sqrt{1-(1-\tilh)^2}.
\end{equation*}

\section{Properties of \texorpdfstring{$\tilde{h}$}{h}}\label{section:tilde h}

This section is devoted to investigating properties of $\tilde{h}$. We begin by proving the sharp upper bound $\tilde{h}\leq 1$, before deriving several lower bounds, discussing a few examples and posing two open questions. Also in this section, all graphs are assumed to be simple and to have no isolated vertices.

\begin{proposition}\label{prop:h1}
    For any graph,
    $$
    \tilde{h}\leq 1,
    $$
with equality if and only if $G$ is bipartite or disconnected.
\end{proposition}

\begin{proof}
    Let $S \subset V$ be a nonempty subset such that $\vol S \leq \vol V/2$.
    Then,
    \begin{align*}
        \RQ(\mathbbm{1}_S) &= \frac{\sum\limits_{w\in V}\frac{1}{\deg w}\left(\sum\limits_{v\in \mathcal{N}(w)}\mathbbm{1}_S(v)\right)^2}{\sum\limits_{w\in V} \deg w\cdot \mathbbm{1}_S(w)^2}\\
        &= \frac{\sum\limits_{w\in V}\frac{1}{\deg w}\cdot e(w,S) ^2}{\sum\limits_{w\in S} \deg w}\\
        &= \frac1{\vol S} \sum\limits_{w\in V}\frac{e(w,S)^2}{\deg w}\\
        &= \tilde h(S).
    \end{align*}
    Since $0 \leq \RQ(f)\leq 1$ for all functions $f\in C(V)\setminus\{\mathbf 0\}$, it follows immediately that $\tilde h(S)\leq 1$, and consequently, $\tilde h \leq 1$.

    Moreover, $\tilde h = 1$ if and only if there exists a nonempty set of vertices $S$ such that $\vol S\leq \vol V/2$ and $\RQ(\mathbbm{1}_S)=1$. As $\mathbbm{1}_S \neq \mathbf{1}$, this can occur if and only if either the normalized Laplacian admits $0$ as an eigenvalue with multiplicity at least $2$, in which case the graph $G$ is disconnected, or the normalized Laplacian has eigenvalue $2$, in which case a connected component of $G$ is bipartite. In the case where the graph is connected, it must therefore be bipartite, and the statement follows.
\end{proof}

\begin{remark}
Proposition \ref{prop:h1} is consistent with the probabilistic interpretation of $\tilde h$ discussed in Remark \ref{rmk:geom0}. In fact, since $\tilde h(S)$ is the probability that a random walk starting in $S$ returns to $S$ after two steps, it is clear that $\tilde h(S)\leq 1$. Moreover, equality can occur only when the random walk returns to $S$ after two steps with probability $1$, which corresponds precisely to disconnected or bipartite cases.
\end{remark}

The main lower bound for $\tilde{h}$ in this section is the following.

\begin{proposition}\label{prop:N(S)}
    For any graph,
    \begin{equation*}
    \tilde{h}\geq \max_{\substack{\emptyset\neq S \subseteq V:\\ \vol S\leq \vol V/2}} \frac{\vol S}{\vol \mathcal{N}(S)}.
\end{equation*} 
\end{proposition}

\begin{remark}\label{rmk:geom} The quantity $\vol S/ \vol \mathcal{N}(S)$ in Proposition~\ref{prop:N(S)} becomes larger when the vertices in $\mathcal{N}(S)$ devote a larger fraction of their degrees to edges connecting to $S$. This is in line with Remark~\ref{rmk:geom0}, and helps us understand the geometric intuition behind $\tilde{h}$. This quantity also appears in Lemma 6.2 in \cite{chung}, which shows that, if $G$ is not a complete graph and $S\subseteq V$, then
$$
\tau^2+(1-\tau^2)\cdot \frac{\vol S}{\vol V}=1-(1-\tau^2)\cdot \frac{\vol \overline{S}}{\vol V}>\frac{\vol S}{\vol \mathcal{N}(S)}.
$$ In particular, if $\vol S\leq \vol V/2$, this implies that
$$
\frac{\tau^2+1}{2}>\frac{\vol S}{\vol \mathcal{N}(S)},
$$ therefore
$$
\frac{\tau^2+1}{2}>\max_{\substack{\emptyset\neq S \subseteq V:\\ \vol S\leq \vol V/2}} \frac{\vol S}{\vol \mathcal{N}(S)}.
$$
Interestingly, Corollary \ref{cor:lower} below will show that $(\tau^2+1)/2\geq \tilde{h}$. 

\end{remark}

\begin{proof}[Proof of Proposition~\ref{prop:N(S)}]
        By Sedrakyan's inequality, given real numbers $u_1,\ldots,u_n$ and positive real numbers $v_1,\ldots,v_n$,
$$
\sum_{i}\frac{u_i^2}{v_i}\geq \frac{\left(\sum_i u_i\right)^2}{\sum_{i} v_i}.
$$

Therefore, given a nonempty set $S\subseteq V$,   
\begin{align*}
    \tilde{h}(S)&=\frac{1}{\vol S}\cdot\sum_{w\in \mathcal{N}(S)}\frac{e(w, S)^2}{\deg w}\\ &\geq \frac{1}{\vol S}\cdot \frac{\left(\sum_{w\in \mathcal{N}(S)}e(w, S)\right)^2}{\sum_{w\in \mathcal{N}(S)} \deg w}\\ &=\frac{\vol S}{\vol \mathcal{N}(S)},
\end{align*}
 implying that
\begin{equation*}
    \tilde{h}\geq \max_{\substack{\emptyset\neq S \subseteq V:\\ \vol S\leq \vol V/2}} \frac{\vol S}{\vol \mathcal{N}(S)}.
\end{equation*} 
\end{proof}

\begin{example}\label{ex:bip}
Consider a bipartite graph, and fix the sets $S$ and $\overline{S}$ that give the bipartition. Then, $\overline{S}=\mathcal{N}(S)$, and $\vol S=\vol \mathcal{N}(S)$. Therefore, by Proposition~\ref{prop:N(S)}, $\tilde{h}\geq 1$, while by Proposition~\ref{prop:h1} we have $\tilde{h}\leq 1$. This implies that $\tilde{h}= 1$, and both bounds are sharp in this case.
\end{example}

\begin{remark}\label{rmk:bip}
    If $G$ is bipartite, then $\lambda_2\geq 2-\lambda_n=0$, implying that $1-\tau=\min\{\lambda_2,2-\lambda_n\}=0$, and therefore $\tau=1$. By Example~\ref{ex:bip}, this implies that $\tilde{h}=\tau=1$ for bipartite graphs.
\end{remark}

Proposition~\ref{prop:N(S)} has the following corollaries.

\begin{corollary}
    If there exists a set $S\subseteq V$ such that $\vol S = \vol V/2$, then
    \[
    \tilde h \geq \frac12.
    \]
\end{corollary}
\begin{proof}
    By Proposition~\ref{prop:N(S)}, we have that
    \begin{align*}
        \tilde h &\ge \tilde h(S)\\
        &\ge \frac{\vol S}{\vol \mathcal N(S)}\\
        &\ge \frac{\vol S}{\vol V}\\
        &= \frac12.
    \end{align*}
\end{proof}

\begin{corollary}\label{cor:regulareven}
    If $G$ is $k$-regular and $n$ is even, then
    \[
    \tilde h \geq \frac12.
    \]
\end{corollary}

The next proposition gives a lower bound for $\tilde h$ whenever there exists a nonempty set $S\subseteq V$ such that $\vol S\leq \vol V/2$ and the vertices in $\mathcal N(S)$ devote the same proportion of their degrees to edges connecting to $S$.

\begin{proposition}\label{prop:equit}
    Assume that there exist a nonempty set $S\subseteq V$ such that $\vol S\leq \vol V/2$ and a constant $\nu\in (0,1]$ such that, for each $w\in \mathcal{N}(S)$, $\nu=e(w,S)/ \deg w$. Then,
$$
\tilde{h}\geq \nu.
$$ 
\end{proposition}

\begin{proof}
    We have that
$$
\tilde{h}\geq \tilde{h}(S)=\frac{1}{\vol S}\cdot \sum_{w\in \mathcal{N}(S)}\frac{e(w, S)^2}{\deg w}=\frac{1}{\vol S}\cdot \nu \cdot \left(\sum_{w\in \mathcal{N}(S)}e(w, S)\right)=\nu.
$$
\end{proof}

In spectral graph theory, it is an open problem to characterize all graphs satisfying 
$$
\lambda_n=\frac{\chi}{\chi-1},
$$
and several classes of graphs for which equality holds are known, including bipartite graphs, complete graphs, and petal graphs \cite{ElphickWocjan2015,Nikiforov2007,Gabriel-colouring,SunDas2020,beers2025end}. The next corollary of Proposition~\ref{prop:equit} provides a lower bound for $\tilde h$ for such graphs.

\begin{corollary}\label{cor:coloring}
    If a graph has chromatic number $\chi$ and is such that $\lambda_n=\chi/(\chi-1)$, then
    $$
    \tilde{h}\geq \frac{1}{\chi-1}.
    $$
\end{corollary}

\begin{proof}
   By Theorem 3.4 in \cite{Gabriel-colouring}, if $\lambda_n=\chi/(\chi-1)$, then for a fixed proper $\chi$-coloring of $G$ with coloring classes $V_1,\ldots,V_\chi$, for each $i\in\{1,\ldots,\chi\}$ and for every vertex $w$,

    \[
    e(w, V_i) = \begin{cases}
        \frac{\deg w}{\chi-1}, &\text{ if } w\notin V_i,\\
        0, &\text{ if }w\in V_i.
    \end{cases}
    \]

Hence, if we fix $i\in\{1,\ldots,\chi\}$, then clearly $\vol V_i\leq \vol V/2$ (since each coloring class is an independent set), and Proposition~\ref{prop:equit} for $\nu=1/(\chi-1)$ gives

$$
 \tilde{h}\geq \frac{1}{\chi-1}.
$$

\end{proof}

\begin{remark}
    For bipartite graphs, $\chi=2$ and $\lambda_n=2=\chi/(\chi-1)$. Hence, in this case, the bound in Corollary~\ref{cor:coloring} is sharp. 
    
    Similarly, for the cycle graph on three vertices, $\chi=3$, $\lambda_n=3/2=\chi/(\chi-1)$, and one can check that $\tilde{h}=1/2$ (cf.\ Example~\ref{ex:complete} below). Therefore, the bound in Corollary~\ref{cor:coloring} is sharp also in this case.
\end{remark}

The next result shows that $\tilde h(S)$ is symmetric with respect to taking the complement whenever $S$ and $\overline S$ have the same volume.

\begin{proposition}
    If a set $S\subseteq V$ is such that $\vol S = \vol V/2$, then
    \[
    \tilde h(S) = \tilde h(\overline S).
    \]
\end{proposition}
\begin{proof}
    Let $S$ be such that $\vol S = \vol V/2$. In this case, also $\vol \overline S = \vol V/2$. Hence,
    \begin{align*}
        \tilde h(\overline S) &= \frac1{\vol \overline S}\cdot \sum_{w\in V}\frac{e(w,\overline S)^2}{\deg w} \\ &= \frac1{\vol \overline S}\cdot \sum_{w\in V}\frac{(\deg w - e(w,S))^2}{\deg w}\\
        &= \frac1{\vol \overline S}\cdot \sum_{w\in V}\biggl( \deg w - 2e(w,S)+\frac{e(w,S)^2}{\deg w}\biggr)\\
        &= \frac1{\vol S}\biggl(\vol V - 2\cdot \frac{\vol V}{2}+\sum_{w\in V}\frac{e(w,S)^2}{\deg w}\biggr)\\
        &= \frac1{\vol S} \sum_{w\in V}\frac{e(w,S)^2}{\deg w}\\
        &= \tilde h(S).
    \end{align*}
\end{proof}

We now consider several examples illustrating the behavior of $\tilde h$. We begin by computing $\tilde h$ for complete graphs.

\begin{example}\label{ex:complete}
    Consider the complete graph $K_n$ on $n\geq 2$ vertices, and let $S\subset V$ with $|S|\leq n/2$. Then,
    \begin{align*}
        \tilde h(S) &= \frac{\sum_{w\in V}e(w,S)^2}{(n-1)\cdot (n-1)|S|}\\
        &= \frac{|S|(|S|-1)^2+(n-|S|)\cdot |S|^2}{(n-1)^2|S|}\\
        &= \frac{|S|(|S|^2-2|S|+1) + (n-|S|)\cdot |S|^2}{(n-1)^2|S|}\\
        &= \frac{-2|S|^2+|S|+n|S|^2}{(n-1)^2|S|}\\
        &= \frac{|S|(n-2)+1}{(n-1)^2}.
    \end{align*}
    If $n$ is even, then $\tilde h(S)$ is maximal when $|S|=n/2$. Hence,
    \begin{align*}
        \tilde h &= \frac{1}2\cdot \frac{n(n-2)+2}{(n-1)^2}\\
        &= \frac12\cdot \frac{(n-1)^2+1}{(n-1)^2}\\
        &=\frac12\biggl(1+\frac1{(n-1)^2}\biggr).
    \end{align*}
    Furthermore, if $n$ is odd, then $\tilde h(S)$ is maximal if $|S|=(n-1)/2$. Hence, in this case,
    \begin{align*}
        \tilde h &= \frac12\cdot \frac{(n-1)(n-2)+2}{(n-1)^2}\\
        &= \frac12 \biggl( 1 - \frac{n-3}{(n-1)^2}\biggr).
    \end{align*}
    Summarizing,
    \[
    \tilde h = \begin{cases}
        \frac12\biggl(1+\frac1{(n-1)^2}\biggr), &\text{ if $n$ is even},\\
        \frac12 \biggl( 1 - \frac{n-3}{(n-1)^2}\biggr), &\text{ if $n$ is odd.}
    \end{cases}
    \]
    In particular, this illustrates that Corollary~\ref{cor:regulareven} does not hold if $n$ is odd.
\end{example}

\begin{example}\label{ex:petals}
\begin{figure}[h]
    \centering
\includegraphics[width=4cm]{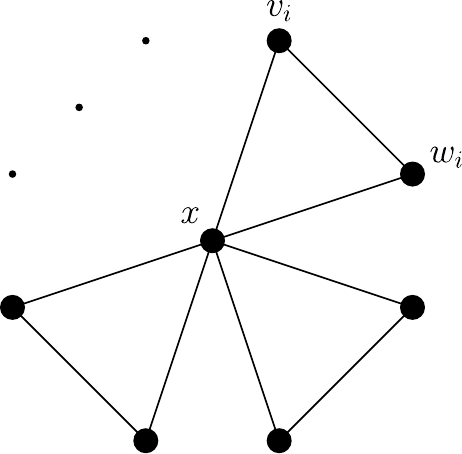}
    \caption{A petal graph.}
    \label{fig:petal}
\end{figure}
    The \emph{$m$-petal graph} on $n=2m+1$ vertices (Figure \ref{fig:petal}) is the graph with vertex set
    $$V = \{x,v_1,\ldots,v_m,w_1,\ldots,w_m\}$$
    and edge set
    \[
     E = \bigcup_{i=1}^m \biggl\{ \{x,v_i\},\{x,w_i\},\{v_i,w_i\} \biggr\}.
    \]

 Let $S$ be a nonempty subset of its vertices satisfying $\vol{S}\leq \vol{V}/2$. Moreover, let $c\in\{0,1\}$ indicate whether the center vertex belongs to $S$, and let $k \in \{0,\ldots,2m\}$ denote the number of peripheral vertices in $S$. Then, $\vol{S}=2mc+2k$. Also, as $\vol{S}\leq \vol{V}/2=3m$, either $c=0$ and $k \leq \lfloor3m/2\rfloor$, or $c=1$ and $k \leq \lfloor m/2\rfloor$.
 
    We compute
    \begin{align*}
    \tilh(S)&=\frac{\sum_{w\in V}\frac{e(w,S)^2}{\deg w}}{\vol{S}}\\
    &=\frac{\frac{k^2}{2m}+k\cdot\frac{(1+c)^2}{2}+(2m-k)\cdot\frac{c^2}{2}}{2mc+2k}\\
    &=\frac{k^2+km+2ckm+2c^2m^2}{4m(k+cm)}.
    \end{align*}

    If $c=0$, this simplifies to $\tilh(S)=(k+m)/(4m)$. If $c=1$, we get $\tilh(S)=(k+2m)/(4m)$ instead. In both cases, $\tilh(S)$ is increasing in $k$, and therefore we can verify that 
$$
\tilh=\frac{\lfloor\frac{m}{2}\rfloor+2m}{4m}.
$$
    If $m$ is even, we obtain $\tilh=5/8$. If $m$ is odd, we get $\tilh=(5m-1)/(8m)$ instead.
\end{example}

By Proposition \ref{prop:h1} we already know that for even cycles we have $\tilh=1$, as they are bipartite. We now compute $\tilh$ for odd cycles.

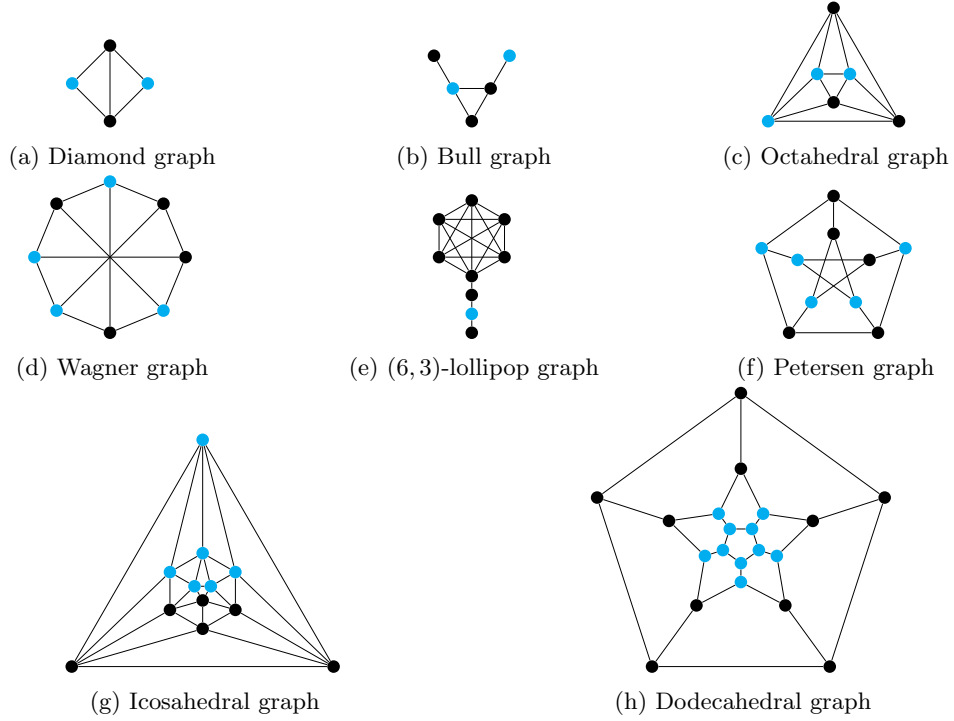
\begin{figure}%[h!]
\centering
\begin{subfigure}{0.32\textwidth}
\centering
\scalebox{0.5}{%
\begin{tikzpicture}
\node[style=circle,fill=black] (0) at (0,0) {};
\node[style=circle,fill=cyan] (1) at (-1,1) {};
\node[style=circle,fill=black] (2) at (0,2) {};
\node[style=circle,fill=cyan] (3) at (1,1) {};

\draw (0) -- (1);
\draw (0) -- (2);
\draw (0) -- (3);
\draw (1) -- (2);
\draw (2) -- (3);
\end{tikzpicture}
}
\subcaption{Diamond graph}
\end{subfigure}
\begin{subfigure}{0.32\textwidth}
\centering
\scalebox{0.5}{%
\begin{tikzpicture}
\node[style=circle,fill=black] (0) at (-1,{sqrt(3)}) {};
\node[style=circle,fill=cyan] (1) at (-1/2,{sqrt(3)/2}) {};
\node[style=circle,fill=black] (2) at (0,0) {};
\node[style=circle,fill=black] (3) at (1/2,{sqrt(3)/2}) {};
\node[style=circle,fill=cyan] (4) at (1,{sqrt(3)}) {};

\draw (0) -- (1);
\draw (1) -- (2);
\draw (1) -- (3);
\draw (2) -- (3);
\draw (3) -- (4);
\end{tikzpicture}
}
\subcaption{Bull graph}
\end{subfigure}
\begin{subfigure}{0.32\textwidth}
\centering
\scalebox{0.5}{%
\begin{tikzpicture}
\node[style=circle, fill=black] (0) at (0,2) {};
\node[style=circle, fill=black] (1) at ({sqrt(12)/2},-1) {};
\node[style=circle, fill=cyan] (2) at ({-sqrt(12)/2},-1) {};
\node[style=circle, fill=black] (3) at (0,-1/2) {};
\node[style=circle, fill=cyan] (4) at ({-sqrt(3)/4},1/4) {};
\node[style=circle, fill=cyan] (5) at ({sqrt(3)/4},1/4) {};

\draw (0) -- (1);
\draw (0) -- (2);
\draw (0) -- (4);
\draw (0) -- (5);
\draw (1) -- (2);
\draw (1) -- (3);
\draw (1) -- (5);
\draw (2) -- (3);
\draw (2) -- (4);
\draw (3) -- (4);
\draw (3) -- (5);
\draw (4) -- (5);
\end{tikzpicture}
}
\subcaption{Octahedral graph}
\end{subfigure}

\begin{subfigure}{0.32\textwidth}
\centering
\scalebox{0.5}{%
\begin{tikzpicture}
\node[style=circle, fill=cyan] (0) at (0,2) {};
\node[style=circle, fill=black] (1) at ({sqrt(2)},{sqrt(2)}) {};
\node[style=circle, fill=black] (2) at (2,0) {};
\node[style=circle, fill=cyan] (3) at ({sqrt(2)},{-sqrt(2)}) {};
\node[style=circle, fill=black] (4) at (0,-2) {};
\node[style=circle, fill=cyan] (5) at (-{sqrt(2)},{-sqrt(2)}) {};
\node[style=circle, fill=cyan] (6) at (-2,0) {};
\node[style=circle, fill=black] (7) at ({-sqrt(2)},{sqrt(2)}) {};

\draw (0) -- (1);
\draw (0) -- (4);
\draw (0) -- (7);
\draw (1) -- (2);
\draw (1) -- (5);
\draw (2) -- (3);
\draw (2) -- (6);
\draw (3) -- (4);
\draw (3) -- (7);
\draw (4) -- (5);
\draw (5) -- (6);
\draw (6) -- (7);
\end{tikzpicture}
}
\subcaption{Wagner graph}
\end{subfigure}
\begin{subfigure}{0.32\textwidth}
\centering
\scalebox{0.5}{%
\begin{tikzpicture}
\node[style=circle, fill=black] (0) at (0,-5/2) {};
\node[style=circle, fill=cyan] (1) at (0,-2) {};
\node[style=circle, fill=black] (2) at (0,-3/2) {};
\node[style=circle, fill=black] (3) at (0,-1) {};
\node[style=circle, fill=black] (4) at ({-sqrt(3)/2},1/2) {};
\node[style=circle, fill=black] (5) at ({sqrt(3)/2},1/2) {};
\node[style=circle, fill=black] (6) at (0,1) {};
\node[style=circle, fill=black] (7) at ({sqrt(12)/4},-1/2) {};
\node[style=circle, fill=black] (8) at ({-sqrt(12)/4},-1/2) {};

\draw (0) -- (1);
\draw (1) -- (2);
\draw (2) -- (3);
\draw (3) -- (4);
\draw (3) -- (5);
\draw (3) -- (6);
\draw (3) -- (7);
\draw (3) -- (8);
\draw (4) -- (5);
\draw (4) -- (6);
\draw (4) -- (7);
\draw (4) -- (8);
\draw (5) -- (6);
\draw (5) -- (7);
\draw (5) -- (8);
\draw (6) -- (7);
\draw (6) -- (8);
\draw (7) -- (8);
\end{tikzpicture}
}
\subcaption{$(6,3)$-lollipop graph}
\end{subfigure}
\begin{subfigure}{0.32\textwidth}
\centering
\scalebox{0.5}{%
\begin{tikzpicture}
\node[style=circle, fill=black] (0) at (0,2) {};
\node[style=circle, fill=cyan] (1) at ({(sqrt(10+2*sqrt(5))))/2},{(sqrt(5)-1)/2}) {};
\node[style=circle, fill=black] (2) at ({(sqrt(10-2*sqrt(5)))/2},{(-sqrt(5)-1)/2}) {};
\node[style=circle, fill=black] (3) at ({-(sqrt(10-2*sqrt(5)))/2},{(-sqrt(5)-1)/2}) {};
\node[style=circle, fill=cyan] (4) at ({-(sqrt(10+2*sqrt(5)))/2},{(sqrt(5)-1)/2}) {};
\node[style=circle, fill=black] (5) at (0,1) {};
\node[style=circle, fill=black] (6) at ({(sqrt(10+2*sqrt(5)))/4},{(sqrt(5)-1)/4}) {};
\node[style=circle, fill=cyan] (7) at ({sqrt(10-2*sqrt(5))/4},{(-sqrt(5)-1)/4}) {};
\node[style=circle, fill=cyan] (8) at ({(-sqrt(10-2*sqrt(5)))/4},{(-sqrt(5)-1)/4}) {};
\node[style=circle, fill=cyan] (9) at ({-(sqrt(10+2*sqrt(5))))/4},{(sqrt(5)-1)/4}) {};
\draw (0) -- (1);
\draw (0) -- (4);
\draw (0) -- (5);
\draw (1) -- (2);
\draw (1) -- (6);
\draw (2) -- (3);
\draw (2) -- (7);
\draw (3) -- (4);
\draw (3) -- (8);
\draw (4) -- (9);
\draw (5) -- (7);
\draw (5) -- (8);
\draw (6) -- (8);
\draw (6) -- (9);
\draw (7) -- (9);
\end{tikzpicture}
}
\subcaption{Petersen graph}
\end{subfigure}

\begin{subfigure}{0.48\textwidth}
\centering
\scalebox{0.5}{%
\begin{tikzpicture}
\node[style=circle, fill=cyan] (0) at (0,4) {};
\node[style=circle, fill=black] (1) at ({sqrt(12)},-2) {};
\node[style=circle, fill=black] (2) at ({-sqrt(12)},-2) {};
\node[style=circle, fill=black] (3) at (0,-1) {};
\node[style=circle, fill=cyan] (4) at ({-sqrt(3)/2},1/2) {};
\node[style=circle, fill=cyan] (5) at ({sqrt(3)/2},1/2) {};
\node[style=circle, fill=cyan] (6) at (0,1) {};
\node[style=circle, fill=black] (7) at ({sqrt(12)/4},-1/2) {};
\node[style=circle, fill=black] (8) at ({-sqrt(12)/4},-1/2) {};
\node[style=circle, fill=black] (9) at (0,-1/4) {};
\node[style=circle, fill=cyan] (10) at ({-sqrt(3)/8},1/8) {};
\node[style=circle, fill=cyan] (11) at ({sqrt(3)/8},1/8) {};

\draw (0) -- (1);
\draw (0) -- (2);
\draw (0) -- (4);
\draw (0) -- (5);
\draw (1) -- (2);
\draw (1) -- (3);
\draw (1) -- (5);
\draw (2) -- (3);
\draw (2) -- (4);
\draw (3) -- (7);
\draw (3) -- (8);
\draw (4) -- (6);
\draw (4) -- (8);
\draw (5) -- (6);
\draw (5) -- (7);
\draw (0) -- (6);
\draw (1) -- (7);
\draw (2) -- (8);
\draw (3) -- (9);
\draw (4) -- (10);
\draw (5) -- (11);
\draw (6) -- (10);
\draw (6) -- (11);
\draw (7) -- (9);
\draw (7) -- (11);
\draw (8) -- (9);
\draw (8) -- (10);
\draw (9) -- (10);
\draw (10) -- (11);
\draw (9) -- (11);
\end{tikzpicture}
}
\subcaption{Icosahedral graph}
\end{subfigure}
\begin{subfigure}{0.48\textwidth}
\centering
\scalebox{0.5}{%
\begin{tikzpicture}
\node[style=circle, fill=black] (0) at (0,4) {};
\node[style=circle, fill=black] (1) at ({(sqrt(10+2*sqrt(5))))},{(sqrt(5)-1)}) {};
\node[style=circle, fill=black] (2) at ({(sqrt(10-2*sqrt(5)))},{(-sqrt(5)-1)}) {};
\node[style=circle, fill=black] (3) at ({-(sqrt(10-2*sqrt(5)))},{(-sqrt(5)-1)}) {};
\node[style=circle, fill=black] (4) at ({-(sqrt(10+2*sqrt(5)))},{(sqrt(5)-1)}) {};
\node[style=circle, fill=black] (5) at (0,2) {};
\node[style=circle, fill=black] (6) at ({(sqrt(10+2*sqrt(5)))/2},{(sqrt(5)-1)/2}) {};
\node[style=circle, fill=black] (7) at ({sqrt(10-2*sqrt(5))/2},{(-sqrt(5)-1)/2}) {};
\node[style=circle, fill=black] (8) at ({(-sqrt(10-2*sqrt(5)))/2},{(-sqrt(5)-1)/2}) {};
\node[style=circle, fill=black] (9) at ({-(sqrt(10+2*sqrt(5))))/2},{(sqrt(5)-1)/2}) {};
\node[style=circle, fill=cyan] (10) at (0,-1) {};
\node[style=circle, fill=cyan] (11) at ({-(sqrt(10+2*sqrt(5)))/4},{-(sqrt(5)-1)/4}) {};
\node[style=circle, fill=cyan] (12) at ({-sqrt(10-2*sqrt(5))/4},{-(-sqrt(5)-1)/4}) {};
\node[style=circle, fill=cyan] (13) at ({(sqrt(10-2*sqrt(5)))/4},{-(-sqrt(5)-1)/4}) {};
\node[style=circle, fill=cyan] (14) at ({(sqrt(10+2*sqrt(5))))/4},{-(sqrt(5)-1)/4}) {};
\node[style=circle, fill=cyan] (15) at (0,-1/2) {};
\node[style=circle, fill=cyan] (16) at ({-(sqrt(10+2*sqrt(5)))/8},{-(sqrt(5)-1)/8}) {};
\node[style=circle, fill=cyan] (17) at ({-sqrt(10-2*sqrt(5))/8},{-(-sqrt(5)-1)/8}) {};
\node[style=circle, fill=cyan] (18) at ({(sqrt(10-2*sqrt(5)))/8},{-(-sqrt(5)-1)/8}) {};
\node[style=circle, fill=cyan] (19) at ({(sqrt(10+2*sqrt(5))))/8},{-(sqrt(5)-1)/8}) {};

\draw (0) -- (1);
\draw (0) -- (4);
\draw (0) -- (5);
\draw (1) -- (2);
\draw (1) -- (6);
\draw (2) -- (3);
\draw (2) -- (7);
\draw (3) -- (4);
\draw (3) -- (8);
\draw (4) -- (9);
\draw (5) -- (12);
\draw (5) -- (13);
\draw (6) -- (13);
\draw (6) -- (14);
\draw (7) -- (10);
\draw (7) -- (14);
\draw (8) -- (10);
\draw (8) -- (11);
\draw (9) -- (11);
\draw (9) -- (12);
\draw (10) -- (15);
\draw (11) -- (16);
\draw (12) -- (17);
\draw (13) -- (18);
\draw (14) -- (19);
\draw (15) -- (16);
\draw (15) -- (19);
\draw (16) -- (17);
\draw (17) -- (18);
\draw (18) -- (19);
\end{tikzpicture}
}
\subcaption{Dodecahedral graph}
\end{subfigure}

\caption{The graphs from Table~\ref{table}. Cyan vertices form sets $S$ for which $\tilh(S)=\tilh$.}
\label{fig:graphs}
\end{figure}

\begin{example}\label{ex:oddcycle}
We consider the odd cycle $C_{2m+1}$ ($m \in \mathbb{N}$), and we let $S$ be a subset of its vertices of size $k>0$ and volume $\vol{S}\leq \vol{V}/2$, that is, satisfying $k \leq m$. Let $l$ be the number of vertices of $C_{2m+1}$ that have two neighbors in $S$, and note that $l\leq k-1$. Then, $2k-2l$ vertices of $C_{2m+1}$ have one neighbor in $S$, and the rest has no neighbors in $S$.

We compute
\begin{align*}
\tilh(S)&=\frac{\sum_{w\in V}\frac{e(w,S)^2}{\deg w}}{\vol{S}}\\
&=\frac{l\cdot\frac{2^2}{2}+(2k-2l)\cdot\frac{1}{2}}{2k}\\
&=\frac{k+l}{2k}.
\end{align*}

Note that $(k+l)/(2k)$ is an increasing function of $l$, therefore $\tilh(S)\leq(2k-1)/2k$. As $(2k-1)/2k$ is an increasing function of $k$, we get $\tilh(S)\leq(2k-1)/2k$.

Hence, $\tilh \leq (2k-1)/2k$. At the same time, by choosing $T$ to be an independent set of size $m$, we get $\tilh(T)=(2k-1)/2k$. Therefore, for odd cycles, $$\tilh=\frac{2m-1}{2m}.$$
\end{example}

We now compute $\tilh$ for a selection of small connected non-bipartite
graphs. The values in Table \ref{table} below
were obtained computationally.\newline 

\begin{table}[h!]
\begin{center}\renewcommand{\arraystretch}{1.3}
\begin{tabular}{|c c c|}
\hline
Graph & $n$ & $\tilh$ \\
\hline\hline
Diamond graph & $4$ & $2/3$ \\ 
\hline
Bull graph & $5$ & $17/24$ \\
\hline
Octahedral graph & $6$ & $7/12$ \\
\hline
Wagner graph & $8$ & $7/9$ \\
\hline
$(6,3)$-lollipop graph & $9$ & $3/4$ \\
\hline
Petersen graph & $10$ & $31/45$ \\
\hline
Icosahedral graph & $12$ & $3/5$ \\
\hline
Dodecahedral graph & $20$ & $7/9$ \\
\hline
\end{tabular}\caption{Values of $\tilh$ for several graphs on $n$ vertices depicted in Figure~\ref{fig:graphs}.}\label{table}
\end{center}
\end{table}

Note that for all graphs in Table \ref{table} and for most graphs in Examples \ref{ex:complete}, \ref{ex:petals}, and \ref{ex:oddcycle}, we have $\tilde h > 1/2$. This observation motivates the following question.
\begin{question}
    Is there a constant $c>0$ such that $\tilde{h}>c$ for all graphs without isolated vertices? If so, what is the largest such constant?
\end{question}

If such a constant $c$ exists, it must be strictly smaller than $1/2$, because complete graphs $K_n$ on an odd number of vertices satisfy $\tilde h < 1/2$, as shown in Example \ref{ex:complete}. Nevertheless, as $n\to\infty$, the value of $\tilde h$ for $K_n$ converges to $1/2$. This observation motivates the following question.
\begin{question}
   Is there a function $f(n)$ such that $\tilde h \geq f(n)$ for all graphs on $n$ vertices, satisfying $\lim_{n\to\infty} f(n) = 1/2$?
\end{question}

\section{Lower bounds for \texorpdfstring{$\tau^2$}{t2}}\label{section:lower}

This section is devoted to proving two sharp lower bounds for $\tau^2$. First, in Theorem~\ref{thm:lower} we establish a general lower bound that does not involve $\tilde h$. As a consequence, we then derive the sharp inequality in Corollary~\ref{cor:lower} which is expressed in terms of $\tilde h$. Again, all graphs are assumed to be simple and to have no isolated vertices.

\begin{theorem}\label{thm:lower}
    For every graph,
    \begin{equation*}
     \tau^2\geq   \max_{\emptyset \neq S\subsetneq V} \left(\frac{\vol V}{\vol \overline{S}}\cdot\left(\frac{1}{\vol S}\cdot\sum_{w\in V}\frac{e(w, S)^2}{\deg w}\right)-\frac{\vol S}{\vol \overline{S}}\right).
    \end{equation*}
    Moreover, the inequality is sharp.
\end{theorem}

The proof of Theorem~\ref{thm:lower} is based on evaluating the Rayleigh Quotient of a function that is constant on $S$ and on its complement. The beginning of the proof is analogous to the classical proof \cite{chung} of the upper bound in the Cheeger inequalities \eqref{eq:Cheeger}.

\begin{proof}[Proof of Theorem~\ref{thm:lower}]
    Fix $\emptyset \neq S\subsetneq V$, let
$\alpha:=\vol S/\vol \overline{S}$, and let $f:V\rightarrow\mathbb{R}$ be such that $f:=1$ on $S$ and $f:=-\alpha$ on $\overline{S}$. Then, $\sum_{v\in V}\deg v\cdot f(v)=0$, implying that 
\begin{equation}\label{eq:lower1}
    \tau^2\geq \textrm{RQ}(f)
    =\frac{\sum\limits_{w\in V}\frac{1}{\deg w}\left(\sum\limits_{v\in \mathcal{N}(w)}f(v)\right)^2}{\sum\limits_{w\in V} \deg w\cdot f(w)^2}.
\end{equation} Since $\alpha\vol \overline{S}=\vol S$, the denominator in \eqref{eq:lower1} can be rewritten as
\begin{equation*}
    \sum\limits_{w\in V} \deg w\cdot f(w)^2=\vol S+\alpha^2 \vol \overline{S}=(1+\alpha)\vol S.
\end{equation*} 

Moreover, the numerator satisfies
\begin{align*}
    \sum\limits_{w\in V}\frac{1}{\deg w}\left(\sum\limits_{v\in \mathcal{N}(w)}f(v)\right)^2&= \sum\limits_{w\in V}\frac{1}{\deg w}\left(\sum\limits_{v\in \mathcal{N}(w)\cap S}1-\sum\limits_{v\in \mathcal{N}(w)\cap \overline{S}}\alpha\right)^2\\
    &=\sum\limits_{w\in V}\frac{1}{\deg w}\biggl(e(w, S)-\alpha e(w, \overline{S})\biggr)^2.
\end{align*}

Note also that
\begin{equation*}
    \vol S=\sum_{v\in S}\deg v=\sum_{w\in V} e(w,S)
\end{equation*}
and, for all $w\in V$,
$$
e(w,\overline{S})=\deg w-e(w,S).
$$

Therefore,

\begin{align*}
  \textrm{RQ}(f)&= \frac{\sum\limits_{w\in V}\frac{1}{\deg w}\biggl(e(w, S)-\alpha e(w, \overline{S})\biggr)^2}{(1+\alpha)\vol S}\\
    &=\frac{\sum\limits_{w\in V}\frac{1}{\deg w}\biggl(e(w, S)-\alpha\left(\deg w-e(w,S)\right)\biggr)^2}{(1+\alpha)\vol S}\\
     &=\frac{\sum\limits_{w\in V}\frac{1}{\deg w}\biggl((1+\alpha)e(w, S)-\alpha\deg w\biggr)^2}{(1+\alpha)\vol S}\\
     &=\frac{\sum\limits_{w\in V}\frac{1}{\deg w}\biggl((1+\alpha)^2e(w, S)^2+\alpha^2(\deg w)^2-2\alpha\deg w(1+\alpha)e(w, S)\biggr)}{(1+\alpha)\vol S}\\
     &=\frac{\sum\limits_{w\in V}\biggl(\frac{(1+\alpha)^2e(w, S)^2}{\deg w}+\alpha^2\deg w-2\alpha(1+\alpha)e(w, S)\biggr)}{(1+\alpha)\vol S}\\
     &=\sum_{w\in V}\frac{(1+\alpha)e(w, S)^2}{\deg w\cdot \vol S}+\sum_{w\in V}\frac{\alpha^2\deg w}{(1+\alpha)\vol S}-\sum_{w\in V}\frac{2\alpha e(w, S)}{\vol S}\\
     &=\frac{(1+\alpha)}{\vol S}\cdot\left(\sum_{w\in V}\frac{e(w, S)^2}{\deg w}\right)+\frac{\alpha^2\vol V}{(1+\alpha)\vol S}-2\alpha.
\end{align*}
Now, observe that $$1+\alpha=1+\frac{\vol S}{\vol \overline{S}}=\frac{\vol \overline{S}+\vol S}{\vol \overline{S}}=\frac{\vol V}{\vol \overline{S}}.$$ Hence,

\begin{equation*}
    \frac{\alpha^2\vol V}{(1+\alpha)\vol S}=\frac{\vol \overline{S}}{\vol V}\cdot \frac{(\vol S)^2}{(\vol \overline{S})^2}\cdot \frac{\vol V}{\vol S}=\frac{\vol S}{\vol \overline{S}}=\alpha,
\end{equation*}and similarly
\begin{equation*}
    \frac{(1+\alpha)}{\vol S}=\frac{\vol V}{\vol S\cdot \vol \overline{S}}.
\end{equation*}

This implies that 
\begin{equation}\label{eq:lower}
    \tau^2\geq \textrm{RQ}(f)=\frac{\vol V}{\vol S\cdot \vol \overline{S}}\cdot\left(\sum_{w\in V}\frac{e(w, S)^2}{\deg w}\right)-\alpha.
\end{equation} This proves the first claim. To show that the inequality is sharp, assume that $G$ is the complete graph. In this case, $\lambda_2=\lambda_n=n/(n-1)$, implying that
$$\tau=\frac{1}{n-1}.$$
If we let $S=\{x\}$, then $\vol S=n-1$ and $\vol \overline{S}=(n-1)^2$. Hence, \eqref{eq:lower} becomes 
\begin{equation*}
    \tau^2\geq \frac{n(n-1)}{(n-1)^3}\cdot\left(\sum_{w\in V\setminus \{x\}}\frac{1}{n-1}\right)-\frac{n-1}{(n-1)^2}=\frac{n}{(n-1)^2}-\frac{n-1}{(n-1)^2}=\frac{1}{(n-1)^2}.
\end{equation*}Equivalently, $\tau\geq 1/(n-1)$. Since we have shown that $\tau= 1/(n-1)$ in this case, the inequality is sharp.
\end{proof}

\begin{remark}\label{rmk:bipartite}
If $G$ is bipartite, then $\tau=1$ by Remark~\ref{rmk:bip}. In this case, the sets $S$ and $\overline{S}$ that give the bipartition satisfy
    \begin{equation*}
        \sum_{w\in V}\frac{e(w, S)^2}{\deg w}=\sum_{w\in \overline{S}}\frac{(\deg w)^2}{\deg w}=\vol \overline{S}.
    \end{equation*}
    Moreover, in this case $\vol S=\vol \overline{S}=\vol V/2$, hence \eqref{eq:lower} becomes
    $$
    \tau^2\geq \frac{\vol V}{\vol S\cdot \vol \overline{S}}\cdot \vol \overline{S}-\frac{\vol S}{\vol \overline{S}}=1,
    $$ or equivalently, $\tau\geq 1$. Hence, the inequality from Theorem~\ref{thm:lower} is sharp also in this case.
\end{remark}

\begin{remark} 
    If $G$ is disconnected, then $\tau=1$, since $\lambda_2 = 0$. In this case, there exists a set $S\subset V$ of volume $\vol S \leq \vol V/2$ that forms the vertex set of a connected component. In this case, Equation~\eqref{eq:lower1} becomes
    \begin{align*}
        \tau^2 &\geq \frac{\vol V}{\vol S\cdot \vol \overline S}\cdot \biggl(\sum_{w\in S}\frac{(\deg w)^2}{\deg w}\biggr) - \frac{\vol S}{\vol \overline S}\\
        &= \frac{\vol S + \vol \overline S}{\vol \overline S}-\frac{\vol S}{\vol \overline S}\\
        &= 1
    \end{align*}
    or equivalently, $\tau\geq 1$. Hence, the inequality from Theorem~\ref{thm:lower} is also sharp in this case.
\end{remark}

\begin{remark}
    The expression from the statement of Theorem~\ref{thm:lower} is symmetric in $S$ and $\overline S$. To see this, consider $f$ as defined in the proof, and define $f'$ by reversing the roles of $S$ and $\overline S$. Then,
    \[
    f' = -\frac{f}{\alpha}.
    \]
    Since $f$ and $f'$ are multiples of one another, they have the same Rayleigh quotient. Hence,
    \begin{align*}
         \RQ(f)&=\frac{\vol V}{\vol \overline{S}}\cdot\left(\frac{1}{\vol S}\cdot\sum_{w\in V}\frac{e(w, S)^2}{\deg w}\right)-\frac{\vol S}{\vol \overline{S}}\\
    &= \RQ(f') \\
    &= \left(\frac{\vol V}{\vol{S}}\cdot\left(\frac{1}{\vol \overline S}\cdot\sum_{w\in V}\frac{e(w,\overline{S})^2}{\deg w}\right)-\frac{\vol \overline S}{\vol {S}}\right).
    \end{align*}

\end{remark}

We are now ready to prove the following sharp lower bound for $\tau^2$ in terms of $\tilde h$, which follows from Theorem~\ref{thm:lower}.

\begin{corollary}\label{cor:lower}
    For every graph,
    \begin{equation*}
     \tau^2\geq2\tilde{h}-1.
    \end{equation*}
    Moreover, the inequality is sharp.
\end{corollary}

\begin{proof}

Fix a nonempty set $S\subseteq V$ such that $\vol S\leq \vol V/2$, or equivalently, $\vol S\leq \vol \overline{S}$. Then, by Theorem~\ref{thm:lower},

\begin{equation*}
    \tau^2\geq \frac{\vol V}{\vol \overline{S}}\cdot \tilde{h}(S)-\frac{\vol S}{\vol \overline{S}}.
\end{equation*} Now, let $\alpha:=\vol S/\vol \overline{S}$. Then, $\alpha\leq 1$ and 
$$
\frac{\vol V}{\vol \overline{S}}=\frac{\vol \overline{S}+\vol S}{\vol \overline{S}}=1+\alpha.
$$ Hence,
$$
\tau^2 \geq (1+\alpha)\tilde{h}(S)-\alpha=\tilde{h}(S)+\alpha(\tilde{h}(S)-1).
$$

Now,  the inequalities in the proof of Proposition~\ref{prop:h1} show that $\tilde{h}(S)\leq 1$ for any $S$. Therefore,  $\tilde{h}(S)-1\leq 0$, and together with the fact that $\alpha \leq 1$, this implies that

$$
\tau^2 \geq \tilde{h}(S)+\alpha (\tilde{h}(S)-1)\geq \tilde{h}(S)+\tilde{h}(S)-1=2\tilde{h}(S)-1.
$$

Maximizing over all nonempty sets $S\subseteq V$ with $\vol S\leq \vol V/2$ gives
$$
\tau^2 \geq 2\tilde{h}-1.
$$
        By Remark~\ref{rmk:bip}, equality holds for bipartite graphs, showing that the inequality is sharp.
\end{proof}

\begin{remark}\label{rmk:petals}
For the $m$-petal graph on $n=2m+1$ vertices, with an even number $m$ of petals, we have $\tilde{h}=5/8$ by Example \ref{ex:petals}, while $\tau^2=1/4$ by \cite{JMZ23}. Hence, the inequality in Corollary~\ref{cor:lower} is sharp also in this case.
\end{remark}

\begin{comment}
   If $m$ is even, let
    \[
    S \coloneqq \{x\} \cup \biggl\{v_i\colon 1\leq i\leq \frac m2\biggr\}.
    \]
       Then,
    \begin{align*}
        \tilde h(S) &= \frac1{\vol S}\cdot \sum_{w\in V} \frac{e(w,S)^2}{\deg w}\\
        &= \frac1{\deg x+m/2\cdot \deg v_1}\biggl( \frac{e(x,S)^2}{\deg x}+ \sum_{i=1}^m \frac{e(v_i,S)^2}{\deg v_i} + \sum_{i=1}^{m/2}\frac{e(w_i,S)^2}{\deg w_i} + \sum_{i=m/2+1}^{m}\frac{e(w_i,S)^2}{\deg w_i}\biggr)\\
        &= \frac1{2m + m/2 \cdot 2}\biggl( \frac{m^2/4}{2m} + m\cdot \frac1{2} + \frac m2 \cdot \frac{4}{2} + \frac m2 \cdot \frac12\biggr)\\
        &= \frac1{3}\biggl( \frac18+\frac12+1+\frac14\biggr)\\
        &= \frac58.
    \end{align*}
    Consequently, $\tilde h \geq 5/8$. Moreover, by Corollary \ref{cor:lower} and since $\tau^2=1/4$ in this case \cite{JMZ23}, we also have the opposite inequality: $\tilde h \leq 5/8$. Therefore, for petal graphs with an even number of petals, $\tilde{h}=5/8$, and the bound from Corollary~\ref{cor:lower} is an equality.
\end{comment}

\begin{remark}
    If $G$ is a connected non-bipartite $d$-regular Ramanujan graph, then by \eqref{eq:Ramanujan},
\begin{equation*}
    \tau^2\leq \frac{4(d-1)}{d^2}.
\end{equation*} Together with Corollary~\ref{cor:lower} this implies that, in this case,
$$
\tilde{h}\leq \frac{1}{2}+\frac{2(d-1)}{d^2}.
$$
\end{remark}

\section{Upper bounds for \texorpdfstring{$\tau^2$}{t2}}\label{section:upper}

Also in this section, all graphs are assumed to be simple and to have no isolated vertices. The main upper bound for $\tau^2$ is given by the following theorem, whose proof is presented in Section~\ref{section:proof}.

\begin{theorem}\label{thm:upper}
    For every graph $G$,
    \begin{equation}\label{eq:upperbound}
     \tau^2\leq \sqrt{1-(1-\tilh)^2}.
    \end{equation}
    Moreover, the inequality is an equality if and only if $G$ is bipartite or disconnected.
\end{theorem}

Applying Bernoulli's inequality in a standard fashion, from Theorem~\ref{thm:upper} one can obtain a weaker upper bound on $\tau^2$ without a square root that is a good approximation when $\tilh$ is close to $1$. 
\begin{corollary}\label{cor:upper}
   For every graph without isolated vertices we have
    \begin{equation*}
     \tau^2\leq 1-\frac{(1-\tilh)^2}{2}.
    \end{equation*}
    Moreover, the inequality is sharp if and only if $G$ is disconnected or bipartite.  
\end{corollary}

\begin{proof}
Bernoulli's inequality states that for any real numbers $r\geq1$ and $x\geq-1$, we have $$(1+x)^r\leq1+rx.$$ Setting $r=1/2$ and $x=-(1-\tilde{h})^2$, and applying Bernoulli's inequality to the upper bound on $\tau^2$ from Theorem~\ref{thm:upper}, we get
\begin{equation}\label{eq:Bern}
    \tau^2\leq\sqrt{1-(1-\tilh)^2}\leq 1-\frac{(1-\tilh)^2}{2},
\end{equation}
proving the first claim.

If $\tau^2=1-(1-\tilh)^2/2$, then from \eqref{eq:Bern} we have in particular $\tau^2=\sqrt{1-(1-\tilh)^2}$. By the equality case from Theorem~\ref{thm:upper}, we know that this equality holds if and only if $G$ is disconnected or bipartite, that is, if and only if $\tilde{h}=1$. Therefore, for $\tau^2=1-(1-\tilh)^2/2$ to hold, it is necessary that $\tilh=1$.

On the other hand, if $\tilh=1$, then $\tau^2=1$ (by the equality case of Theorem~\ref{thm:upper}), and also $1-(1-\tilh)^2/2=1$. Therefore, $\tilh=1$ is both a necessary and sufficient condition for $\tau^2=1-(1-\tilh)^2/2$. Thus, the inequality from the statement is an equality if and only if $G$ is disconnected or bipartite, as claimed.
\end{proof}

%We note that, although the bound stated in Corollary~\ref{cor:upper} is weaker than the one in Theorem~\ref{thm:upper}, it is not uncommon to use the version of Corollary \ref{cor:upper} to avoid the square root.

%an alternative formulation of the Cheeger inequality that avoids the square root, namely
%$$
%1-\sqrt{1-h^2}\leq \lambda_2.
%$$
We note that while the bound in Corollary~\ref{cor:upper} is weaker than the one in Theorem~\ref{thm:upper}, it is not uncommon to use an expression avoiding the square root in the Cheeger inequality
$$
1-\sqrt{1-h^2}\leq \lambda_2.
$$
Compare, for example, the quantity
 $h^2/2$ in Theorem~2.2 in Chung's book \cite{chung},
with the quantity $1-\sqrt{1-h^2}$
in Theorem~2.3 in the same book \cite{chung}.

\section{Proof of Theorem~\ref{thm:upper}}\label{section:proof}
This section is dedicated to proving Theorem~\ref{thm:upper}.
Much of the proof of Theorem~\ref{thm:upper} uses similar techniques as one of the proofs of the lower bound of the Cheeger inequalities in \eqref{eq:Cheeger}, as presented for example by Luca Trevisan in \cite{trevisan2017lecture} (other previous work featuring similar arguments includes \cite{mohar1989isoperimetric} and \cite{chung2007four}). However, our proof has several additional technicalities due to the form of the Rayleigh Quotient of the operator $M$.\\

First, we introduce a few results that will be helpful for the proof. We start with the following basic property of fractions.

\begin{remark}\label{rmk:minmax}
Let $p,r \geq 0$ and $q,s>0$. Then,
$$\min\biggl\{\frac{p}{q},\frac{r}{s}\biggr\}\leq\frac{p+r}{q+s}\leq \max\biggl\{\frac{p}{q},\frac{r}{s}\biggr\}.$$
\end{remark}

Another ingredient is the following lemma, which is essentially Fact~5.2 from \cite{trevisan2017lecture}, differing only in the inequality inside of the probability. Although the proof is almost identical, we reproduce it here for the sake of completeness.  
\begin{lemma}\label{lem:xyexey}
    Let $X$ and $Y$ be random variables such that $\mathbb{P}(Y>0)=1$. Then,
    $$\mathbb{P}\biggl(\frac{X}{Y}\geq \frac{\mathbb{E}X}{\mathbb{E}Y}\biggr)>0.$$
\end{lemma}
\begin{proof}[Proof of Lemma~\ref{lem:xyexey}]
    By linearity of expectation, we see that
    \[
    \mathbb{E}\biggl(X-\frac{\mathbb{E}(X)}{\mathbb{E}(Y)}\cdot Y\biggr) = \mathbb{E}(X)-\frac{\mathbb{E}(X)}{\mathbb{E}(Y)}\cdot \mathbb{E}(Y)=0.
    \]
   Therefore,
    \[
    \mathbbm{P}\biggl( X \geq \frac{\mathbb{E}(X)}{\mathbb{E}(Y)}\cdot Y\biggr)>0,
    \]
    and the claim follows by dividing both sides of the internal inequality by $Y$, which we may do because $Y>0$ almost surely.
\end{proof}

We are now ready to go through the proof of Theorem~\ref{thm:upper}, which is divided into several claims in order to make the argument easier to follow.

\begin{proof}[Proof of Theorem~\ref{thm:upper}]

\renewcommand{\qedsymbol}{$\blacksquare$}

Recall that, by \eqref{eq:tau2}, we have the following characterization:

    $$\tau^2=\max_{\substack{f \in C(V)\setminus\{ \mathbf{0}\}: \\ \sum_{w\in V}\deg w \cdot f(w)=0}} \textrm{RQ}(f)=\max_{\substack{f \in C(V)\setminus\{ \mathbf{0}\}: \\ \sum_{w\in V}\deg w \cdot f(w)=0}}\frac{\sum\limits_{w\in V}\frac{1}{\deg w}\left(\sum\limits_{v\in \mathcal{N}(w)}f(v)\right)^2}{\sum\limits_{w\in V} \deg w\cdot f(w)^2}.$$

    We first find a convenient function whose Rayleigh Quotient gives an upper bound to $\tau^2$, similarly to Lemma~4.5 in \cite{trevisan2017lecture}.

\begin{claim}\label{cl:f}
        There exists a function $f\in C(V)\setminus \{\mathbf 0\}$ such that
        \[
        \tau^2 \leq \RQ(f), 
        \]
        satisfying
        \[
        \vol{{\mathrm{supp}(f)}} \in \biggl(0,\frac{\vol V}{2}\biggr]
        \quad \text{ and }
        \quad
        1\in \range(f) \subset [0,1].
        \]
    \end{claim}

    \begin{proof}[Proof of Claim~\ref{cl:f}]
    Let $g \in C(V)\setminus\{\mathbf{0}\}$ be a function satisfying \[\sum_{w\in V}\deg w \cdot g(w)=0 \quad \text{and} \quad 
    \tau^2=\mathrm{RQ}(g).\] Notice that, as a consequence, $g$ attains at least one positive and one negative value.\newline 
    
    For any $c \in \mathbb{R}$, we have $\mathrm{RQ}(g-c\cdot \mathbf{1})\geq \mathrm{RQ}(g)$. In fact, for $c\neq 0$,
    \begin{align*}
        \RQ(g-c\cdot \mathbf 1) &= \frac{\langle g-c\cdot \mathbf 1,M(g-c\cdot \mathbf{1}) \rangle}{\langle g - c \cdot \mathbf 1,g-c\cdot \mathbf 1\rangle}\\
        &= \frac{\langle g,Mg \rangle + \langle -c\cdot \mathbf 1,M(-c\cdot \mathbf{1}) \rangle}{\langle g,g\rangle + \langle - c \cdot \mathbf 1, -c\cdot \mathbf 1\rangle}\\
        &= \frac{\langle g,Mg \rangle + c^2 \langle \mathbf 1,M\mathbf{1} \rangle}{\langle g,g\rangle + c^2\cdot \langle \mathbf 1,  \mathbf 1\rangle}\\
        &\ge \min\biggl\{ \frac{\langle g,Mg \rangle}{\langle g,g\rangle}, \frac{c^2 \langle \mathbf 1,M\mathbf{1} \rangle}{c^2\cdot \langle \mathbf 1,  \mathbf 1\rangle}\biggr\}\\
        &= \min \bigl\{ \RQ(g),\RQ(\mathbf 1)\bigr\}\\
        &= \RQ(g),
    \end{align*}
    where we have used Remark~\ref{rmk:minmax} for the inequality, and afterward used the fact that $1=\RQ(\mathbf{1})$ is the largest eigenvalue of $M$.\\

    We now pick $c$ such that $$\vol{\{v: g(v)-c \geq0\}}\geq \frac{\vol V}{2} \quad \text{and} \quad \vol{\{v: g(v)-c \leq 0\}}\geq \frac{\vol V}{2}.$$ To see that this can be done, first recall that $g$ attains only finitely many values, say $g_1 < \ldots <g_k$, which then satisfy 
    $$\vol{\{v: g(v)-g_1 \geq0\}}=\vol{V}\quad \text{and} \quad \vol{\{v: g(v)-g_k \leq 0\}}=\vol{V}.$$ Hence, there exists $i\in\{1\ldots,k\}$ such that 
    $$\vol{\{v: g(v)-g_i \geq0\}}\geq\frac{\vol V}{2}\quad \text{and} \quad \vol{\{v: g(v)-g_{i+1} \leq0\}}\geq\frac{\vol V}{2}.$$ If neither $g_i$ nor $g_{i+1}$ can be chosen as $c$, then 
$$\vol\left\{v: g(v)-\frac{g_i+g_{i+1}}{2} <0\right\}<\frac{\vol V}{2}\quad \text{and} \quad \vol{\left\{v: g(v)-\frac{g_i+g_{i+1}}{2} >0\right\}}<\frac{\vol V}{2},$$ which contradicts that these two volumes add up to $\vol{V}$, because $(g_i+g_{i+1})/2$ is not in the range of $g$. Hence, $g_i$ or $g_{i+1}$ can be chosen as $c$. \newline 

In particular, this choice of $c$ also implies that $$\vol{\{v: g(v)-c <0\}} \leq \frac{\vol V}{2}\quad \text{and} \quad \vol{\{v: g(v)-c > 0\}}\leq \frac{\vol V}{2}.$$ 

Set now $\tilde{g}:=g-c\cdot \mathbf{1}$, $\tilde{g}^+:=\max{\{\tilde{g},\mathbf{0}\}}$, and $\tilde{g}^-:=\max{\{-\tilde{g},\mathbf{0}\}}$, so that $\tilde{g}=\tilde{g}^+-\tilde{g}^-$. Then, $$\vol{\mathrm{supp}(\tilde{g}^+)} \leq \frac{\vol V}{2}\quad \text{and} \quad \vol{\mathrm{supp}(\tilde{g}^-)} \leq \frac{\vol V}{2}.$$ 
Furthermore, note that
    $
    \supp\bigl(\tilde g^+\bigr) \cap \supp \bigl(\tilde g^-\bigr) = \emptyset.$\\
    
    We now claim that at least one $\tilde{f} \in \{\tilde{g}^+,\tilde{g}^-\}$ satisfies $\tilde{f}\neq \mathbf{0}$ and $\mathrm{RQ}(\tilde{g})\leq \mathrm{RQ}(\tilde{f})$.\newline

    First, note that $\tilde{g}\neq 0$ since $g$ attains at least two different values. Now,
    if $\tilde{g}^+=0$, then $\RQ(\tilde g)=\RQ(\tilde{g}^-)$, proving the claim.
    Analogously, if $\tilde{g}^-=0$, then $\RQ(\tilde g)=\RQ(\tilde{g}^+)$.
    
    If both $\tilde{g}^+$ and $\tilde{g}^-$ are nonzero, then
    \begin{align*}
        \RQ(\tilde g) &= \frac{\langle \tilde{g}^+ - \tilde{g}^-,M( \tilde{g}^+ - \tilde{g}^-)\rangle}{\langle \tilde{g}^+ - \tilde{g}^-,\tilde{g}^+ - \tilde{g}^-\rangle}\\
        &= \frac{\langle \tilde{g}^+,M\tilde{g}^+\rangle + \langle \tilde{g}^-,M\tilde{g}^-\rangle - \langle \tilde{g}^+,M\tilde{g}^-\rangle- \langle \tilde{g}^-,M\tilde{g}^+\rangle}{\langle \tilde{g}^+,\tilde{g}^+\rangle + \langle \tilde{g}^-,\tilde{g}^-\rangle-2\langle \tilde{g}^+,\tilde{g}^-\rangle}\\
        &= \frac{\langle \tilde{g}^+,M\tilde{g}^+\rangle + \langle \tilde{g}^-,M\tilde{g}^-\rangle - \langle \tilde{g}^+,M\tilde{g}^-\rangle- \langle \tilde{g}^-,M\tilde{g}^+\rangle}{\langle \tilde{g}^+,\tilde{g}^+\rangle + \langle \tilde{g}^-,\tilde{g}^-\rangle}\\
        &\overset{(i)}{\leq} \frac{\langle \tilde{g}^+,M\tilde{g}^+\rangle + \langle \tilde{g}^-,M\tilde{g}^-\rangle}{\langle \tilde{g}^+,\tilde{g}^+\rangle + \langle \tilde{g}^-,\tilde{g}^-\rangle}\\
        &\overset{(ii)}{\leq} \max\biggl\{\frac{\langle \tilde{g}^+,M\tilde{g}^+\rangle}{\langle \tilde{g}^+,\tilde{g}^+\rangle},\frac{\langle \tilde{g}^-,M\tilde{g}^-\rangle}{\langle \tilde{g}^-,\tilde{g}^-\rangle}\biggr\}\\
        &= \max\bigr\{\RQ\bigl(\tilde g^+\bigr),\RQ\bigl(\tilde g^-\bigr)\bigr\}.
    \end{align*}
    In $(i)$, we used the fact that all entries of $M$ are nonnegative by Equation~\eqref{eq:Mentries}, and so are all entries of $\tilde g^+$ and $\tilde g^-$. Furthermore, in $(ii)$, we used Remark~\ref{rmk:minmax}.

    Finally, we set $$f:=\frac{\tilde{f}}{\max{\{\tilde{f}(v): v \in V\}}},$$
    and we note that this is well-defined. The function $f: V \rightarrow \mathbb{R}$ then satisfies
    \[
    \vol{{\mathrm{supp}(f)}} \in \left(0,\frac{\vol V}{2}\right] \quad \text{ and } \quad 1\in \range(f) \subset [0,1],
    \]
    as well as the chain of (in)equalities
    \[
    \tau^2 = \RQ(g) \leq \RQ(g - c \cdot \mathbf{1}) \leq \RQ(\tilde f) = \RQ(f).
    \] This proves Claim \ref{cl:f}.
    \end{proof}
    Next, we use $f$ from Claim~\ref{cl:f} to find a set $S \subset V$ such that $\vol{S} \in (0, \vol{V}/2]$ and
\begin{equation}\label{eq:UB}
    \mathrm{RQ}(f) \leq \sqrt{1-(1-\tilh(S))^2}. 
    \end{equation} 
    To this end, we introduce a random subset $S_t$ of $V$, following Chapter 5 of Trevisan \cite{trevisan2017lecture}.
    Due to the more complicated form of our Rayleigh quotient, our proof then deviates from Trevisan's argument and becomes more involved.

    \begin{claim}\label{cl:S}
        There is a set $S$ satisfying
        \[
        \RQ(f) \leq \sqrt{1-\left(1-\tilde{h}(S)\right)^2} \quad \text{ and }\quad \vol S \leq \frac{\vol V}2.
        \]
    \end{claim}

    \begin{proof}[Proof of Claim~\ref{cl:S}]
    To show the existence of such $S$, we consider the random set $$S_t:=\{v: f(v)\geq t\},$$ where $t$ is a random variable distributed as $t^2 \sim U[0,1]$. We thus have $$\mathbb{P}(v \in S_t)=\mathbb{P}(f(v)\geq t)=\mathbb{P}(f(v)^2\geq t^2)=f(v)^2.$$ Moreover, for two vertices $u,v \in V$, we have
    \begin{equation}\label{eq:prob}
        \mathbb{P}(u,v\in S_t)=\min\{f(u)^2,f(v)^2\} \quad \text{ and } \quad \mathbb{P}((u\in S_t) \lor (v\in S_t))=\max\{f(u)^2,f(v)^2\}.
    \end{equation}
    As $\max \{f(v): v\in V\}=1$ by Claim \ref{cl:f}, we also have that $\mathbb{P}(S_t = \emptyset)=0$. Moreover, since $\mathbb{P}(t=0)=0$, almost surely $S_t \subset \mathrm{supp}(f)$. Therefore, since  $\vol{{\mathrm{supp}(f)}} \in (0, \vol{V}/2]$ by Claim \ref{cl:f}, it follows that, almost surely, $\vol{S_t} \in (0, \vol{V}/2]$. Moreover, since 
    $$
   \vol S_t=\sum_{v\in V}\deg v\cdot  \mathbbm{1}_{S_t}(v),$$ we have that
    $$
  \mathbb{E}\left( \vol S_t\right)=\sum_{v\in V}\deg v\cdot  \mathbb{E}\left(\mathbbm{1}_{S_t}(v)\right)=\sum_{v\in V}\deg v\cdot  \mathbb{P}\left(v\in S_t\right)=\sum_{v \in V} \deg v \cdot f(v)^2.
    $$

    Now, to prove \eqref{eq:UB}, we first bound its left-hand side by rewriting the Rayleigh quotient in a form that is suitable for the Cauchy-Schwarz inequality as follows.
    \begin{align}
    \RQ(f)&=
    \frac{\sum\limits_{w\in V}\frac{1}{\deg w}\left(\sum\limits_{v\in \mathcal{N}(w)}f(v)\right)^2}{\sum\limits_{w\in V} \deg w\cdot f(w)^2} \nonumber \\
    &= \frac{\sum_{\substack{w\in V\\ (u,v)\in \mathcal N(w)^2}} \frac{f(u)f(v)}{\deg w}}{\sum\limits_{w\in V} \deg w\cdot f(w)^2} \nonumber
    \\&= \frac{\sum_{\substack{w\in V \nonumber \\ (u,v)\in \mathcal N(w)^2}} \frac{\min\{f(u),f(v)\}\cdot \max\{f(u),f(v)\}}{\deg w}}{\sum\limits_{w\in V} \deg w\cdot f(w)^2} \nonumber \\
    &\leq \frac{\biggl(\sum_{\substack{w\in V\\ (u,v)\in \mathcal N(w)^2}} \frac{\min\{f(u),f(v)\}^2}{\deg w}\biggr)^{1/2}\cdot \biggl(\sum_{\substack{w\in V\\ (u,v)\in \mathcal N(w)^2}} \frac{\max\{f(u),f(v)\}^2}{\deg w}\biggr)^{1/2} }{\sum\limits_{w\in V} \deg w\cdot f(w)^2} \label{line1},
\end{align}
where in~\eqref{line1} we applied the Cauchy-Schwarz inequality to the vectors
\[
\biggl( \frac{\min\{f(u),f(v)\}}{\sqrt{\deg w}} \biggr)_{\substack{w\in V\\ (u,v) \in \mathcal N(w)^2}} \quad \text{ and }\quad \biggl( \frac{\max\{f(u),f(v)\}}{\sqrt{\deg w}} \biggr)_{\substack{w\in V\\ (u,v) \in \mathcal N(w)^2}}.
\]

We can now use \eqref{eq:prob}, together with linearity of expectation, to further bound $\RQ(f)$ by rewriting \eqref{line1} as follows:      
\begin{align}
    \RQ(f)&\leq \frac{\biggl(\sum_{\substack{w\in V\\ (u,v)\in \mathcal N(w)^2}} \frac{\mathbb{P}(u,v\in S_t)}{\deg w}\biggr)^{1/2}\cdot \biggl(\sum_{\substack{w\in V\\ (u,v)\in \mathcal N(w)^2}} \frac{\mathbb{P}(u\in S_t \vee v \in S_t)}{\deg w}\biggr)^{1/2} }{\mathbb E(\vol S_t)} \nonumber \\
    &= \frac{\biggl(\sum_{\substack{w\in V\\ (u,v)\in \mathcal N(w)^2}} \frac{\mathbb E(\mathbbm{1}_{\{u,v\in S_t\}})}{\deg w}\biggr)^{1/2}\cdot \biggl(\sum_{\substack{w\in V\\ (u,v)\in \mathcal N(w)^2}} \frac{\mathbb{E}(\mathbbm{1}_{\{u\in S_t \vee v \in S_t\}})}{\deg w}\biggr)^{1/2} }{\mathbb E(\vol S_t)}\nonumber \\
    &= \frac{\biggl(\mathbb E\bigl(\sum_{\substack{w\in V\\ (u,v)\in \mathcal N(w)^2}} \frac{\mathbbm{1}_{\{u,v\in S_t\}}}{\deg w}\bigr)\biggr)^{1/2}\cdot \biggl(\mathbb{E}\bigl(\sum_{\substack{w\in V\\ (u,v)\in \mathcal N(w)^2}} \frac{\mathbbm{1}_{\{u\in S_t \vee v \in S_t\}}}{\deg w}\bigr)\biggr)^{1/2} }{\mathbb E(\vol S_t)} \label{line2}.
\end{align}
Now, note that
    $$\mathbbm{1}((u\in S_t)\lor(v\in S_t))=\mathbbm{1}(u,v\in S_t)+\mathbbm{1}((u\in S_t)\dot\lor(v\in S_t)),$$
    where $\dot\lor$ denotes the exclusive or, and that
    $$\sum\limits_{u,v\in \mathcal{N}(w)}\mathbbm{1}(u,v \in S_t)=\left(\sum\limits_{v\in \mathcal{N}(w)}\mathbbm{1}(v \in S_t)\right)^2=e(w,S_t)^2.$$
    Furthermore, 
    \begin{align*}
    \sum\limits_{(u,v)\in \mathcal{N}(w)^2}\mathbbm{1}(u \in S_t \dot\vee v \in S_t)&=\sum\limits_{(u,v)\in \mathcal{N}(w)^2}\mathbbm{1}((u \in S_t) \land (v \notin S_t))+\sum\limits_{(u,v)\in \mathcal{N}(w)^2}\mathbbm{1}((u \notin S_t) \land (v \in S_t))\\
    &={2}e(w,S_t)\cdot e(w,\bar{S_t}).
    \end{align*}
    
    The right-hand side of \eqref{line2} can thus be further rewritten in the following way:
    
\begin{align}
    \RQ(f)&\leq \frac{\biggl(\mathbb E \bigl(\sum_{w\in V} \frac{e(w,S_t)^2}{\deg w}\bigr)\biggr)^{1/2}\cdot \biggl(\mathbb{E}\bigl(\sum_{w\in V} \frac{e(w,S_t)^2 + 2e(w,S_t)\cdot e(w,\overline{S_t})}{\deg w}\bigr)\biggr)^{1/2} }{\mathbb E(\vol S_t)} \nonumber \\
    &= \frac{\biggl(\mathbb E \bigl(\sum_{w\in V} \frac{e(w,S_t)^2}{\deg w}\bigr)\biggr)^{1/2}\cdot \biggl(\mathbb{E}\bigl(\sum_{w\in V} \frac{e(w,S_t)^2 + 2e(w,S_t)(\deg w - e(w,S_t))}{\deg w}\bigr)\biggr)^{1/2} }{\mathbb E(\vol S_t)} \nonumber \\
    &= \frac{\biggl(\mathbb E \bigl(\sum_{w\in V} \frac{e(w,S_t)^2}{\deg w}\bigr)\biggr)^{1/2}\cdot \biggl(\mathbb{E}\bigl(\sum_{w\in V}( \frac{-e(w,S_t)^2}{\deg w} + 2e(w,S_t))\bigr)\biggr)^{1/2} }{\mathbb E(\vol S_t)} \nonumber \\
    &= \label{line3} \frac{\biggl(\mathbb E \bigl(\sum_{w\in V} \frac{e(w,S_t)^2}{\deg w}\bigr)\biggr)^{1/2}\cdot \biggl(-\mathbb{E}(\sum_{w\in V} \frac{e(w,S_t)^2}{\deg w}) + 2\cdot\mathbb{E}(\vol S_t)\biggr)^{1/2} }{\mathbb E(\vol S_t)}.
\end{align}
    
    We now define two random variables \[X_t\coloneqq \sum_{w\in V} \frac{e(w,S_t)^2}{\deg w} \quad \text{ and }Y_t\coloneqq \vol S_t,\] 
  and we note that $\mathbb{P}(Y_t>0)=1$, since we have seen that $\vol{S_t} \in (0, \vol{V}/2]$ almost surely. Using this notation, we can rewrite \eqref{line3} to obtain
    \begin{align*}
    \RQ(f)&\leq
    \frac{\biggl(\mathbb E \bigl(X_t)\biggr)^{1/2}\cdot \biggl(-\mathbb{E}(X_t) + 2\cdot\mathbb{E}(Y_t)\biggr)^{1/2} }{\mathbb E(Y_t)}\\
    &=\sqrt{-\frac{\mathbb{E}(X_t)^2}{\mathbb{E}(Y_t)^2}+2\frac{\mathbb{E}(X_t)}{\mathbb{E}(Y_t)}}\\
    &=\sqrt{1-\left(1-\frac{\mathbb{E}(X_t)}{\mathbb{E}(Y_t)}\right)^2}.
    \end{align*}
    By Lemma~\ref{lem:xyexey}, there is some $t_0 \in [0,1]$ such that
\begin{equation}\label{eq:Etot0}
\frac{\mathbb E(X_t)}{\mathbb E(Y_t)}\leq \frac{X_{t_0}}{Y_{t_0}}.
\end{equation}
Using \eqref{eq:Etot0} together with the fact that the function $x\mapsto \sqrt{1-(1-x)^2}$ is increasing on the interval $[0,1]$, we thus immediately get
\begin{align*}
    \sqrt{1-\left(1-\frac{\mathbb{E}(X_t)}{\mathbb{E}(Y_t)}\right)^2}
    &\leq \sqrt{1-\left(1-\frac{X_{t_0}}{Y_{t_0}}\right)^2}\\
    &= \sqrt{1-\Bigl(1-\frac{1}{\vol S_{t_0}}\sum_{w\in V} \frac{e(w,S_{t_0})^2}{\deg w}\Bigr)^2}\\
    &=\sqrt{1-(1-\tilde{h}(S_{t_0}))^2}.
\end{align*}
This proves the existence of $S=S_{t_0}$ satisfying \eqref{eq:UB} and $\vol S \leq \vol V/2$.
\end{proof}
Combining Claims~\ref{cl:f} and \ref{cl:S} together with the fact that the function $x\mapsto \sqrt{1-(1-x)^2}$ is increasing on the interval $[0,1]$, we have
    \[\tau^2 \leq \RQ(f)\leq \sqrt{1-\left(1-\tilde{h}(S)\right)^2} \leq \sqrt{1-\left(1-\tilde{h}\right)^2},\]
concluding the proof of the upper bound on $\tau^2$.
\newline

In the remainder of the proof, we analyze the equality case.

\begin{claim}\label{cl:equality}
    The inequality from~\eqref{eq:upperbound} is an equality if and only if $G$ is bipartite or disconnected.
\end{claim}

\begin{proof}[Proof of Claim~\ref{cl:equality}.]

If the inequality is sharp, then the Cauchy–Schwarz bound in \eqref{line1} must itself be attained sharply. This occurs if and only if, for every $w \in V$ and every pair $(u,v) \in \mathcal N(w)^2$, either $$\min\{f(u),f(v)\}=t\cdot \max\{f(u),f(v)\}\quad \text{or}\quad t \cdot \min\{f(u),f(v)\} = \max\{f(u),f(v)\},$$
for some $t\geq 0$. In particular, if $u=v$, then necessarily $f(u)=f(v)$, which implies $t=1$ and thus
\begin{equation}\label{eq:CSeq}
f(u)=f(v) \text{ for all } w\in V \text{ and all } (u,v)\in \mathcal N(w)^2.
\end{equation}
If $G$ is disconnected, then $\tilh = \tau^2 = 1$, and hence the bound in \eqref{eq:upperbound} is an equality.

Therefore, we may assume that $G$ is connected.
Fix a vertex $v_0 \in V$. Observe that, by \eqref{eq:CSeq}, all vertices in the neighborhood $\mathcal{N}(v_0)$ necessarily assume a common function value; we denote this value by $a_1$. Moreover, the vertices in the neighborhood of $\mathcal{N}(v_0)$, which in particular includes $v_0$ itself, must assume another common function value; we denote this value by $a_0$.

This reasoning can be extended inductively to the entire graph. Let $S_0$ denote the set of vertices at an even distance from $v_0$, and let $S_1$ denote the set of vertices at an odd distance from $v_0$.
By construction, every vertex in $S_0$ has at least one neighbor in $S_1$ and vice versa, since $v_0$ has at least one neighbor in $S_1$, and every vertex $w\neq v_0$ has a neighbor $v$ satisfying $d(v,v_0) = d(w,v_0)-1.$ One obtains that all vertices in $S_0$ must attain the common function value $a_0$, and analogously, all vertices in $S_1$ attain the common function value $a_1$.

Recall that the function $f$ attains at least two function values by construction (one being equal to zero), from which it follows that $a_0\neq a_1$. Furthermore, by Equation~\eqref{eq:CSeq}, no vertex can have a neighbor in $S_0$ and a neighbor in $S_1$.

We conclude that $S_0$ and $S_1$ are independent sets, and that $G$ is bipartite, with bipartition $(S_0,S_1)$.

Conversely, if $G$ is disconnected or bipartite, then $\tau^2 = \tilh =1$, in which case Equation~\eqref{eq:upperbound} is an equality.
\end{proof}

This concludes the proof of Theorem~\ref{thm:upper}.
\renewcommand{\qedsymbol}{$\square$}
\end{proof}

\section*{Tool and computational resource disclosure}
The authors used Python to perform computations related to some of the examples presented in this paper. The authors also used a large language model as a tool for language polishing and for identifying possible inconsistencies or errors in preliminary drafts.

\section*{Funding}
Raffaella Mulas is supported by the Dutch Research Council (NWO) through the grant VI.Veni.232.002.
The research of Jan Petr is funded by the Deutsche Forschungsgemeinschaft (DFG, German Research Foundation) – 542321564.

\bibliographystyle{plain} 

\bibliography{Bibliography}

\end{document}